\documentclass[12pt]{amsart}
\usepackage[utf8]{inputenc}
\usepackage{euscript}
\usepackage{a4wide}
\usepackage[T1]{fontenc}
\usepackage{lmodern}
\usepackage{amsfonts}
\usepackage{amscd}
\usepackage[matrix,arrow,curve]{xy}
\usepackage{amssymb, amsthm}
\usepackage{amsmath}
\usepackage{pdfpages}
\usepackage{graphicx}
\usepackage{mathtools}
\usepackage{geometry}
\usepackage{tikz}
\usetikzlibrary{calc}

\usepackage{mathrsfs, amssymb,amsthm, amsfonts, amsmath}

\newcounter{cequation}[section]

\theoremstyle{plain}
\newtheorem{thm}[cequation]{Theorem}
\newtheorem{lm}[cequation]{Lemma}
\newtheorem{thm-rem}[cequation]{Theorem-Remark}
\newtheorem{cor}[cequation]{Corollary}
\newtheorem{st}[cequation]{Proposition}
\newtheorem{q}[cequation]{Question}

\theoremstyle{definition}
\newtheorem{defn}[cequation]{Definition}
\newtheorem{ex}[cequation]{Example}
\newtheorem{n}[cequation]{Remark}
\theoremstyle{remark}

\makeatletter\@addtoreset{equation}{section}
\makeatother

\makeatother

\newcommand{\Aut}{\operatorname{Aut}}
\newcommand{\Ime}{\operatorname{Im}}
\newcommand{\Nm}{\operatorname{Nm}}
\newcommand{\ord}{\operatorname{ord}}
\newcommand{\Bir}{\operatorname{Bir}}
\newcommand{\id}{\operatorname{id}}
\newcommand{\rk}{\operatorname{rk}}

\newcommand{\SO}{\operatorname{SO}}
\newcommand{\GL}{\operatorname{GL}}

\newcommand{\lcm}{\operatorname{lcm}}
\newcommand{\im}{\operatorname{Im}}
\newcommand{\Ker}{\operatorname{Ker}}
\newcommand{\Mat}{\operatorname{Mat}}
\newcommand{\Gal}{\operatorname{Gal}}
\newcommand{\Br}{\operatorname{Br}}
\newcommand{\Pic}{\operatorname{Pic}}
\newcommand{\Char}{\operatorname{char}}
\newcommand{\tr}{\operatorname{tr.deg}}
\newcommand{\mumu}{{\boldsymbol{\mu}}}

\usepackage{graphicx}
\graphicspath {{picture/}}
\DeclareGraphicsExtensions{.pdf}

\date{}

\title{Finite Subgroups of Automorphism Groups of Severi--Brauer Varieties of Prime Degree}

\author{Alexandra Sonina}

\address{
Steklov Mathematical Institute of Russian Academy of Sciences, 8 Gubkina St., Moscow 119333, Russia
}

\email{sasha-sonina@mail.ru}

\begin{document}

\begin{abstract}

We classify finite subgroups of automorphism groups of non-trivial Severi–Brauer varieties of dimension $q-1$, where $q\geqslant 3$ is a prime number, over an arbitrary field. We also construct families of examples, namely, for every consistent set of finite groups, we construct a field together with a non-trivial Severi–Brauer variety over that field such that every group in the set acts on the constructed variety. Additionally, we show that non-trivial Severi–Brauer varieties of dimension $q-1$, where $q \geqslant 3$ is a prime number, over a field of characteristic not equal to $q$ are not $G$-birationally rigid. 

\end{abstract}

\maketitle
\tableofcontents

\section{Introduction}

\textit{A Severi--Brauer variety} of dimension $n - 1$ over a field $\mathbb{K}$ is an algebraic variety $X$ which becomes isomorphic to $\mathbb{P}^{n-1}_{\overline{\mathbb{K}}}$ after base change to an algebraic closure $\overline{\mathbb{K}}$ of $\mathbb{K}$. 
A Severi--Brauer variety is called \textit{non-trivial} if it is not isomorphic to the projective space~$\mathbb{P}^{n-1}_{\mathbb{K}}$ over the base field~$\mathbb{K}$. There is a well-known bijection between Severi--Brauer varieties of dimension~$n - 1$ and central simple algebras of degree~$n$, if one fixes a field~$\mathbb{K}$ and a positive integer~$n$. Furthermore, this bijection preserves automorphism groups (see e.g.~\cite[~Chapter~5]{G-S}). 

One may ask the following question: ``Which finite groups can act regularly and faithfully on a Severi--Brauer variety?''. Unfortunately, the answer to this question is not very meaningful. By any division algebra~$D$ and any finite group~$G$, one can construct the matrix algebra~$\Mat_{|G|}(D)$, which is always a central simple algebra. One can notice that~$G$ is embedded into~$\GL_{|G|}(D)$ by the left action of~$G$ on~$D^{|G|}$, and the image of~$G$ in~$\GL_{|G|}$ intersects with~$\mathbb{K}^*$, which is embedded as diagonal matrices, trivially. Therefore~$G$ appears as a subgroup of~$\Aut(\Mat_{|G|}(D))$. By Wedderburn's theorem (see e.g.~{\cite[~Theorem~2.1.3]{G-S}}) for every finite-dimensional central simple algebra $A$ over a field~$\mathbb{K}$ there exist an integer $n \geqslant 1$ and a division algebra $D \supset \mathbb{K}$ such that $A$ is isomorphic to a matrix algebra~$\Mat_n(D)$. That is why one may wonder which finite groups can act regularly and faithfully on Severi--Brauer varieties that correspond to division algebras. Such Severi--Brauer varieties are called \textit{minimal}. From a geometric point of view, one can describe them as the Severi--Brauer varieties that have no non-trivial twisted linear Severi--Brauer subvarieties (see e.g.~\cite[Corollary~5.3.5]{G-S}). If~$p$ is a prime number, by Wedderburn's theorem, every non-trivial Severi--Brauer variety of dimension~$p-1$ is minimal. We focus on those varieties. The question was first proposed in this generality by Anna Savelyeva. In the article~\cite{ASav}, she provided Theorem \ref{Sav}, where the conditions for finite subgroups in the case of dimension~$p-1$ are given. 


Let $\mumu_n$ denote the cyclic group of order~$n$. Let $q$ be a prime number. Suppose $n$ is divisible only by prime numbers $p$ such that $p$ is congruent to $1$ modulo $q$. Then a semidirect product $\mumu_n \rtimes \mumu_q$ is called {\sf balanced} if $\mumu_n \rtimes \mumu_q$ has trivial center. An important remark is that for a fixed pair $(n,q)$, a balanced semidirect product might not be unique. For more details, see Section \ref{balance}.

The question about finite groups acting on Severi--Brauer varieties first appeared in~\cite{Sh20}, which provides a complete classification of finite subgroups of the automorphism group for surfaces over fields of characteristic zero.

    \begin{thm}[{\cite[Theorem 1.3]{Sh20}}]
      Let $\mathbb{K}$ be a field of characteristic $0$. Let $S$ be a non-trivial Severi--Brauer surface over $\mathbb{K}$. Let~$G$ be a finite subgroup of~$\Aut(S)$, then there exists a positive integer~$n$ such that~$G$ is isomorphic to a subgroup of~$\mumu_3 \times (\mumu_n \rtimes \mumu_3)$, where the semidirect product is balanced.
    \end{thm}

    \begin{thm}[{\cite[Theorem 1.3]{Sh20}}]
      \label{Shr2}
      For any $n$ such that the semidirect product $\mumu_n \rtimes \mumu_3$ is balanced, there exists a field~$\mathbb{K}\subset \overline{\mathbb{Q}}$ and a non-trivial Severi--Brauer surface~$S$ over $\mathbb{K}$ such that~${\mumu_3 \times (\mumu_n \rtimes \mumu_3)}$ is a subgroup of $\Aut(S)$.
    \end{thm}

Moreover, a later article \cite{Sh20B} provided a classification of finite subgroups of the birational automorphism group for surfaces over fields of characteristic zero.

\begin{thm}[{\cite[Proposition 3.7 and Corollary 4.5]{Sh20B}}]
     \label{Shr3}
   Let~$S$ be a non-trivial Severi--Brauer surface over a field of characteristic zero. Then any finite subgroup of~$\Bir(S)$ is either  conjugate to a subgroup of~$\Aut(S)$ or isomorphic to a subgroup of~$\mumu^3_3$.
\end{thm}

The following theorem answers the question of which groups can act on non-trivial Severi--Brauer varieties of dimension $q - 1$, where $q$ is an odd prime number.

\begin{thm}[{\cite[Theorem 1.4]{ASav}}]
    \label{Sav}
    Let~$X$ be a non-trivial Severi--Brauer variety of dimension~$q - 1$ over a field~$\mathbb{K}$, where~$q \geqslant 3$ is prime and $q \neq \Char(\mathbb{K})$. Let~$G$ be a finite subgroup of~$\Aut(X)$, then there exists a positive integer~$n$ such that~$G$ is isomorphic to a subgroup of~$\mumu_q \times (\mumu_n \rtimes \mumu_q)$, where the semidirect product is balanced.
\end{thm}



One natural next question is the following: ``Can all possible finite groups act regularly and faithfully on one fixed Severi--Brauer variety?''. Unfortunately, Theorem~\ref{consistancy} shows that many groups cannot act on a Severi--Brauer variety simultaneously, even though each of them may act individually on some different Severi--Brauer varieties. In other words, there are compatibility conditions that prevent some groups from acting on the same Severi--Brauer variety.



We will call a set of groups that agree with each other in a certain sense a consistent balanced collection (see Definition \ref{hastypelong}). The first theorem, which we are going to prove, claims that for any consistent balanced collection of finite groups, there is an example of a non-trivial Severi--Brauer variety such that all groups from that collection act on the variety. Moreover, in the case of dimension $2$, its group of birational automorphisms contains $\mumu^3_3$.
   
\begin{thm}
    \label{example}
     Let~$q \geqslant 3$ be a prime number. Let $\frak{A} = \{\mumu_q \times (\mumu_n \rtimes_{\chi} \mumu_q)\}$ be a consistent balanced collection. Then there exists a field~$\Bbbk \subset \overline{\mathbb{Q}}$ and a non-trivial Severi~--~Brauer variety~$X$ of dimension~$q - 1$ over~$\mathbb{K} = \Bbbk(t)$, where $t$ is a transcendental variable, such that~$\Aut(X)$ contains all groups~${\mumu_q \times (\mumu_n \rtimes_{\chi} \mumu_q)} \in \frak{A}$. In particular, for any group $\mumu_q \times (\mumu_n \rtimes_{\chi} \mumu_q)$, where~$\mumu_n \rtimes_{\chi} \mumu_q$ is a balanced semidirect product, there is a non-trivial Severi--Brauer variety $X$ such that~${{\mumu_q \times (\mumu_n \rtimes_{\chi} \mumu_q) \subset \Aut(X)}}$. Moreover, if $q = 3$, then~${\mumu^3_3 \subset \Bir(X)}$. 
\end{thm}

The second goal of this paper is to generalize the result of Theorem \ref{example} to the case of positive characteristic. In this case, not all of the finite groups described in Theorem~\ref{Sav} can actually be realized. And even more, in the case of positive characteristic, groups that act on the variety form a special consistent balanced collection, which will be called a Frobenius consistent balanced collection (see Definition \ref{Ftype}). 

\begin{thm}
    \label{maincon}
     Let~$l$ be any prime number, and let~$\mathbb{K}$ be a field of characteristic~$l$. Let~$X$ be a non-trivial Severi--Brauer variety of dimension~$q - 1$ over a field~$\mathbb{K}$, where~$q \geqslant 3$ is prime and~$l \neq q$. Let $$\frak{A} = \{\mumu_n \rtimes_{\chi} \mumu_q \mid \mumu_n \rtimes_{\chi} \mumu_q  \subset \Aut(X)\}$$ and $$D = \{ n \mid n \text{ is not divisible by } q \text{ and } \mumu_n \subset \Aut(X)\}.$$ Then $\frak{A}$ is a Frobenius consistent balanced collection and the set $D$ is $q$-adically constant with respect to the order of $l$ (see~Definition~\ref{qconst}). In particular, all integers $n \in D$ are not divisible by $l$. And moreover, if $G$ is a finite subgroup of $\Aut(X)$ then either $G$ is a cyclic group or there exist~$n$ and $\chi$ such that $\mumu_n \rtimes_{\chi}\mumu_q \in \frak{A}$ and $G$ is a subgroup of~${\mumu_q \times (\mumu_n \rtimes_{\chi}\mumu_q)}$.
     
     
\end{thm}

Note that if for every $n \in D$ one has a group $\mumu_n \rtimes_{\chi} \mumu_q \in \frak{A}$, then the condition that $D$ is $q$-adically constant with respect to the order of $l$ follows from the condition that $\frak{A}$ is a Frobenius consistent balanced collection.  

In the case when $l = q$, the assertion of Theorem \ref{maincon} becomes in some sense more complicated but in some other sense much easier.

\begin{thm}
\label{mainconp}
Let~$q \geqslant 3$ be a prime number, and let~$\mathbb{K}$ be a field of characteristic~$q$. Let~$X$ be a non-trivial Severi--Brauer variety of dimension~$q - 1$ over a field~$\mathbb{K}$. Let $$\frak{A} = \{\mumu_n \rtimes_{\chi} \mumu_q \mid \mumu_n \rtimes_{\chi} \mumu_q  \subset \Aut(X)\}$$ and $$D = \{ n \mid n \text{ is not divisible by } q \text{ and } \mumu_n \subset \Aut(X)\}.$$ Then $\frak{A}$ is a Frobenius consistent balanced collection and the set $D$ is $q$-adically constant with respect to the order of $q$ (see~Definition~\ref{qconst}). And moreover, if $G$ is a finite subgroup of~$\Aut(X)$ then either $G$ is an abelian group or there exist~$n$ and $\chi$ such that $G$ is isomorphic to $\mumu_n \rtimes_{\chi} \mumu_q$. In the case when $G$ is abelian, $G$ is either isomorphic to some $\mumu_n$ where $n \in D$ or $G$ is isomorphic to $\mumu^N_q$ for some positive $N$.
\end{thm}
 
   \begin{n}
   We should warn the reader that there is no simple example of a non-trivial Severi--Brauer variety $X$ over the base field $\mathbb{K}$ such that $\dim(X) + 1 = \Char{\mathbb{K}}$. Indeed, if~$\mathbb{K}$ is perfect, there is no such non-trivial Severi--Brauer variety. This follows from~\cite[~Theorem~9.2.4]{G-S} and the fact that the module of absolute differentials of a perfect field $\mathbb{K}$ is trivial.
   \end{n}
   
   In particular, Theorems \ref{maincon} and \ref{mainconp} show that there is no example of a non-trivial Severi--Brauer variety over a field of positive characteristic, such that all possible cyclic groups $\mumu_n$ act simultaneously, like in the case of characteristic $0$. But some more subtle examples still exist.
   
   

  
   Our last theorem states that if groups in any collection of finite groups form a consistent balanced Frobenius collection, then there is an example of a non-trivial Severi--Brauer variety such that all groups from that collection act on the variety.

\begin{thm}
    \label{main}
 Let~$l$ be any prime number. Let~$q \geqslant 3$ be a prime number and~${l \neq q}$. Let~${\frak{A} = \{\mumu_q \times (\mumu_n \rtimes_{\chi} \mumu_q)\}}$ be a Frobenius consistent balanced collection. Then there exists a field~$\Bbbk \subset \overline{\mathbb{F}_l}$ and a non-trivial Severi--Brauer variety~$X$ of dimension~$q-1$ over~$\mathbb{K} = \Bbbk(t)$, where $t$ is a transcendental variable, such that $\Aut(X)$ contains all groups~${\mumu_q \times (\mumu_n \rtimes_{\chi} \mumu_q)} \in \frak{A}$.
 \end{thm}
 
 In the case when $l = q$, there are basically the same examples as in the case $l \neq q$, but additionally, we claim that there is no bound for the integer $N$ in the notation of Theorem~\ref{mainconp}.

\begin{thm}
    \label{mainp}
Let~$q \geqslant 3$ be a prime number. Let~${\frak{A} = \{\mumu_n \rtimes_{\chi} \mumu_q\}}$ be a Frobenius consistent balanced collection. Then there exists a field~$\Bbbk \subset \overline{\mathbb{F}_q}$ and a non-trivial Severi~--~Brauer variety~$X$ of dimension~$q-1$ over~$\mathbb{K} = \Bbbk(t)$, where $t$ is a transcendental variable, such that $\Aut(X)$ contains all groups~$\mumu_n \rtimes_{\chi} \mumu_q \in \frak{A}$ and for any positive integer $N$, the group~$\mumu_q^N$ also acts on $X$.
 \end{thm}

The theorems above consider only cases where the dimension of varieties is at least 2. The following example shows there is no way to get something similar, even to Theorem~\ref{Sav}, in the case of Severi--Brauer curves.

\begin{ex}
Consider a conic $C$ given by the equation~${x^2 + y^2 + z^2 = 0}$ on the real projective plane. Since it is a conic without real points, it is a non-trivial Severi--Brauer curve. Notice that $\Aut(C) \supset \SO(3,\mathbb{R)}$. For any integer~$n$, one has a subgroup of~$\Aut(C)$, which is isomorphic to $\mumu_n$. And even more, there is a subgroup of~$\Aut(C)$ isomorphic to the icosahedral group $\mathfrak{A}_5$, which cannot be represented as a semidirect product.
\end{ex}

One may wonder whether there are infinite sets of prime numbers $p$ such that groups~$\mumu_p$ act on the Severi--Brauer varieties described in Theorem \ref{main}. We will give a positive answer in Proposition \ref{infin}.

There is a question that still requires an answer. In Theorems \ref{example}, \ref{main} and \ref{mainp}, one has~$\tr\mathbb{K} = 1$ over~$\mathbb{Q}$ in the first case, and over~$\mathbb{F}_l$ in the second case, where $\mathbb{K}$ is the field over which the example of a Severi--Brauer variety was constructed. So, one may wonder whether there is only a finite set of finite groups which act on a non-trivial Severi--Brauer variety $X$ over the base field~$\mathbb{K}$, such that~$\tr_{\mathbb{L}}(\mathbb{K}) = 0$, where $\mathbb{L}$ is either~$\mathbb{Q}$ or~$\mathbb{F}_{l}$? Proposition \ref{trfin} answers this question negatively in positive characteristic, but still, there is the following question:

\begin{q}

Let $X$ be a non-trivial Severi--Brauer variety over a field $\mathbb{K}$, such that~$\mathbb{K} \subset \overline{\mathbb{Q}}$. Is it true that there is only a finite set of finite groups that act on $X$?

\end{q}

The structure of the paper is as follows: in Section \ref{prel}, some preliminaries and general lemmas, mostly about field extensions, are given; then in Section \ref{balance}, balanced semidirect products are discussed. In Section~\ref{char0}, Theorem~\ref{example} is proven, i.e., examples over fields of characteristic $0$ are constructed. In Sections~\ref{criter} and~\ref{exchar}, we focus on the case where the base field has positive characteristic. First, in Section \ref{criter}, we provide more restrictions on finite subgroups; we also provide an improved version of Theorem~\ref{Sav}, and second, in Section~\ref{exchar}, examples over the base field of positive characteristic are constructed. In Section~\ref{com}, some additional remarks are provided. In Appendix~\ref{Ap}, we show that if $G$ is a finite group acting on a Severi–Brauer variety $X$, then $X$ cannot be~$G$-birationally rigid.

\smallskip
\textbf{Acknowledgements.}
I would like to express my sincere gratitude to Constantin Shramov for proposing the problem, for numerous indispensable discussions, and for his patience during the development of this text. I am thankful to Vladislav Levashev, Sergey Gorchinskiy, and Magnus Ridder Olsen for several helpful discussions and explanations. I am grateful to Anna Savelyeva for her careful reading and many helpful suggestions. I am also very grateful to Andrey Trepalin and Artem Avilov for their valuable comments on a previous version of this paper. Additionally, I thank Sergey Tikhonov for his careful reading and helpful comments on the last version of the text.

This work was performed at the Steklov International Mathematical Center and partially supported by the Ministry of Science and Higher Education of the Russian Federation (agreement no. 075-15-2025-303) and by the Theoretical Physics and Mathematics Advancement Foundation «BASIS».

\section{Preliminaries}
\label{prel}
Let~$A$ be an algebra, and denote by~$A^*$ its group of units. The automorphism group of a Severi–Brauer variety admits the following classical description.
 
\begin{lm}[{see e.g. \cite[Theorem E on page 266]{Chat} or \cite[Lemma 4.1]{ShV}}]
    \label{Aut}
    Let~$X$ be a Severi--Brauer variety over a field~$\Bbbk$ corresponding to a central simple algebra~$A$. Then~${\Aut(X) \cong A^* / \Bbbk^*}$.

\end{lm}

Let~$L/\Bbbk$ be a Galois extension with the Galois group isomorphic to~$\mumu_n$. Choose an element~$a \in \Bbbk^*$ and a generator~$\sigma$ of~$\Gal(L/\Bbbk)$. Then one can associate with~$\sigma$ and~$a$ a cyclic algebra~$A$, which is a central simple algebra over~$\Bbbk$ of degree~$n$, see \cite[2.5]{G-S} or~\cite[Exercise~3.1.6]{GSh}. Explicitly,~$A$ is generated over~$\Bbbk$ by~$L$ and an element~$\alpha$ subject to relations~$\alpha^n = a$ and
$$\lambda \alpha =\alpha \sigma(\lambda), \lambda \in L.$$

One denotes this algebra by~$A = (L/\Bbbk, \sigma, a)$.

\begin{lm} [{see e.g. \cite[Exercise 3.1.6(i)]{GSh}}]
    \label{nm}
    Suppose that in the above notation the element~$a$ is not contained in the image of the Galois norm~$\Nm_{L/\Bbbk}$. Then~$A$ is not isomorphic to a matrix algebra.
\end{lm}

Let us mention some facts about fields extensions. 

\begin{thm}[{see e.g. \cite[Chapter 6, Theorem 9.1]{Leng}}]
     \label{Leng9}
    Let~$\Bbbk$ be a field,~$n \geqslant 2$ is an integer number,~$a \in \Bbbk$ and~$a \neq 0$. Suppose that~$a \notin \Bbbk^p$ for all prime numbers~$p$ which divide~$n$, and~$a \notin -4\Bbbk^4$ if~$n$ is divisible by~$4$. Then the polynomial~$T^n - a$ is irreducible in~$\Bbbk[T]$.
\end{thm}

\begin{lm}
     \label{leng}
    Let~$L \supset \Bbbk$ be a field extension such that~$\left[L : \Bbbk\right] = q$ is a prime number. Let~$\theta \in L$. Let~$p \neq q$ be a prime number and~$\theta^p \in \Bbbk$. Then~$\theta^p \in \Bbbk^p$.
\end{lm}

\begin{proof}
    By Theorem \ref{Leng9} either~$\theta^p \in \Bbbk^p$ or~$x^p - \theta^p$ is irreducible over~$\Bbbk$. In the latter case one has~$\left[\Bbbk(\theta) : \Bbbk\right] = p$ and~$\Bbbk \subset \Bbbk(\theta) \subset L$, so~$q$ is divisible by~$p$, a contradiction.
\end{proof}

\begin{lm}
    \label{n}
     Let~$L \supset \Bbbk$ be a field extension such that~$\left[L : \Bbbk\right] = q$ is a prime number. Let~$n = pm$ be some integer, where~$p$ is some prime number such that~$q \neq p$. Let~$\theta \in L$ and~$\theta^n \in \Bbbk$. Then~$\theta^n \in \Bbbk^{p}$.
\end{lm}

\begin{proof}
   By Lemma \ref{leng} applied to $\theta^m$ one immediately gets that $\theta^{n} \in \Bbbk^{p}$. 
\end{proof}

\begin{lm}
\label{p=l}
Let $L \supset \Bbbk$ be a field extension such that~$\left[L : \Bbbk\right] = q$ is a prime number. Let~$\Char(\Bbbk) = l > 0$ and $q \neq l$. Let~$\theta \in L$ such that~$\theta^l \in \Bbbk$. Then~$\theta \in \Bbbk$.
\end{lm}

\begin{proof}
Denote $\theta^l = \alpha \in \Bbbk$. 

If $\alpha \in \Bbbk^l$, then $$\theta^l - \alpha = \theta^l - \alpha_0^l = (\theta - \alpha_0)^l,$$ for some $\alpha_0 \in \Bbbk$. Therefore $\theta = \alpha_0 \in \Bbbk$.

If $\alpha \notin \Bbbk^l$, then by Theorem \ref{Leng9} a polynomial $T^l - \alpha$ is irreducible. Then $[\Bbbk(\theta):\Bbbk] = l$, but $$q = \left[L : \Bbbk\right] = \left[L : \Bbbk(\theta)\right] \cdot \left[\Bbbk(\theta) : \Bbbk\right] = l \cdot  \left[L : \Bbbk(\theta)\right],$$ which is impossible since $q \neq l$. 
\end{proof}

\begin{lm}
\label{root}
Let $L \supset \Bbbk$ be a field extension such that~$\left[L : \Bbbk\right] = q$ is a prime number. Let~$n = \prod^m_{i=1}p^{r_i}_i$ be some integer, where~$p_i$ are some prime numbers such that~$q \neq p_i$. Let~$\theta \in L$ such that~$\theta^n \in \Bbbk$ and $\theta^s \notin \Bbbk$ for all $1 \leqslant s < n$. Then~$\theta^n \in \Bbbk^{n}$. 
\end{lm}

\begin{proof}
Note that by Lemma \ref{p=l} applied to $\theta^{\frac{n}{p_i}}$ one has that $p_i \neq \Char(\Bbbk)$ for all $1 \leqslant i \leqslant m$. 

Firstly, let us prove that $\theta^n \in \Bbbk^{p^{r_i}_i}$ for all $1 \leqslant i \leqslant m$. Denote $$t_{p_i} = \frac{\prod\limits^m_{j=1}p^{r_j}_j}{p^{r_i}_i},$$ and $$\theta_{p_i} = \theta^{t_{p_i}}.$$

Suppose that $\theta^n \in \Bbbk^{p_i^s}$, but $\theta^n \in \Bbbk^{p_i^{s+1}}$ for some integer $s$ such that $0 \leqslant s < r_i$. Then there is $\alpha \in \Bbbk$ such that $\theta^n = \alpha^{p_i^s}$ and $\alpha \notin \Bbbk^{p_i}$. Then $$\left(\frac{\theta^{p_i^{r_i - s}}_{p_i}}{\alpha}\right)^{p^s_i} - 1 = 0.$$ Therefore there is $\xi_{p_i^t}$ such that $$\theta^{p_i^{r_i - s}}_{p_i} = \xi_{p_i^t}\alpha,$$ where $\xi_{p_i^t}$ is a primitive root of unity of degree $p^{t}_i$ and $1\leqslant t \leqslant s$. 

Since $\theta^{p_i^{r_i - s}}_{p_i} \notin \Bbbk$, then $\xi_{p_i^t} \notin \Bbbk$ and therefore there is no primitive root of unity of degree~$p^{t}_i$ in $\Bbbk$. But all primitive roots of unity of degree $p^{t}_i$ lie in $L$.

Consider a polynomial $f(T) = T^{p^t_i} - \alpha \in L[T]$. Since $p_i \neq l$ this polynomial is separable and since $\alpha \notin \Bbbk^{p_i}$ it is an irreducible polynomial. Suppose $L(y)$ is a splitting field of the polynomial $f(T)$, where $y$ is a root of the polynomial $f(T)$. Since $f(T)$ is separable then all roots of $f(T)$ are of the form $\xi y$, where $\xi$ is a root of unity of degree $p_i^t$. Since there are $p_i^t$ different roots of $f(T)$, all roots of unity, including primitive ones, are met as $\xi$. Consider a splitting field of the polynomial $f(T)$ taken as a polynomial with coefficients in $\Bbbk$. This field is $\Bbbk(y)$. Both $L$ and $\Bbbk(y)$ lie inside $L(y)$, so we can consider $F = L \cap \Bbbk(y)$. Note that $\Bbbk \subset F$ and also all roots of unity of degree $p^t_i$ lie in $F$. Since $\xi_{p_i^t} \notin \Bbbk$ we have that $\Bbbk \subsetneq F$. Since $\left[L : \Bbbk\right] = q$, then~$F = L$. But $\left[\Bbbk(y) : \Bbbk\right] = p_i^t$ and $q \neq p_i$, so $L \nsubseteq \Bbbk(y)$, a contradiction since $F$ must be a subset of $\Bbbk(y)$.


Therefore one gets that there is $\alpha_{p_i}$ such that $$\left(\frac{\theta^{t_{p_i}}}{\alpha_{p_i}}\right)^{p_i^{r_i}} - 1 = 0,$$ therefore $$\frac{\theta^{t_{p_i}}}{\alpha_{r_i,p_i}} = \xi_{p^{r_i}_i},$$ where $\xi_{p_i^{r_i}}$ is a root of unity of degree $p^{r_i}_i$. By multiplying all equations, one has that 
 \begin{equation}
\frac{\theta^{\sum\limits^m_{i = 1}t_{p_i}}}{\prod\limits^{m}_{i = 1}\alpha_{p_i}} = \xi_{n}.
\label{roo}
\end{equation}
 Since $\sum^m_{i = 1}t_{p_i}$ is not divisible by any $p_i$ for $1 \leqslant i \leqslant m$, one can take a suitable power of equation \eqref{roo} and get that $\theta = C \xi_n$ for some $C \in \Bbbk$ and so $\theta^n \in \Bbbk^{n}$. 
\end{proof}

\begin{thm}[{Natural Irrationalities, see e.g. \cite[Chapter 6, Theorem 1.12]{Leng}}]
    \label{NI}
    Let~$K/F$ be a finite degree Galois extension and let~$L/F$ be an arbitrary field extension, such that there is a field $U \supset K,L$. Then:
    \begin{enumerate}
        \item [(a)] The field extension~$KL/L$ is Galois.
        \item [(b)] The restriction map~$r: \Aut(KL/L) \to \Aut (K/ K \cap L)$ is an isomorphism.
        \item [(c)] One has~$\left[KL:L\right] = \left[K:K \cap L\right]$.
    \end{enumerate}
\end{thm}

\begin{lm}
    \label{sur}
    Let~$L \supset \Bbbk$ be a finite Galois extension such that $\Bbbk \neq L$. Let~$t$ be a transcendental variable. Then~$t \not\in \Ime(\Nm_{L(t)/\Bbbk(t)})$. In particular, $\Nm_{L(t)/\Bbbk(t)}$ is not a surjection.
\end{lm}
\begin{proof}
    Suppose that $$\Nm_{L(t)/\Bbbk(t)}(l) = t$$ for some~$l \in L(t)$. One can rewrite~$l = \frac{p(t)}{q(t)}$, where~$p(x),q(x) \in L[x]$. By Theorem \ref{NI} one has~$\Gal(L/\Bbbk) = \Gal(L(t)/\Bbbk(t))$. Let $$n = \left[L:\Bbbk\right] = \left[L(t): \Bbbk(t)\right],$$ and let $$\Gal(L(t)/\Bbbk(t)) = \{\sigma_1,\ldots, \sigma_n\}.$$ One has $$\Nm_{L(t)/\Bbbk(t)}(x) = \sigma_1(x) \cdot \sigma_2(x) \cdot \ldots \cdot \sigma_{n}(x)$$ for any~$x \in L(t)$. Let~$f(t) = a_mt^m + \ldots + a_1t + a_0 \in L[t]$ be any polynomial and~${\sigma \in \Gal(L(t)/\Bbbk(t))}$. Notice that $$\sigma(f(t)) = \sigma(a_mt^m +\ldots+ a_1t + a_0) = \sigma(a_m)t^m +\ldots+\sigma(a_1)t + \sigma(a_0),$$ because~$t \in \Bbbk(t)$. Let us denote $$\sigma(f)(t) = \sigma(a_m)t^m +\ldots+\sigma(a_1)t + \sigma(a_0).$$ Then $$\Nm_{L(t)/\Bbbk(t)}(l) = \frac{\sigma_1(p)(t) \cdot \sigma_2(p)(t) \cdot \ldots\cdot \sigma_{n}(p)(t)}{\sigma_1(q)(t) \cdot \sigma_2(q)(t) \cdot \ldots\cdot \sigma_{n}(q)(t)} = t,$$ so $$\sigma_1(p)(t) \cdot \sigma_2(p)(t) \cdot \ldots\cdot \sigma_{n}(p)(t) = t \cdot \sigma_1(q)(t) \cdot \sigma_2(q)(t) \cdot \ldots \cdot \sigma_{n}(q)(t).$$ Let~$\deg(p) = k$ and~$\deg(q) = m$. Then the leading term of the polynomial on the left is $$\sigma_1(a_k)\cdot \sigma_2(a_k)\cdot \ldots\cdot \sigma_{n}(a_k) \cdot t^{kn},$$ where $a_k$ is the leading coefficient of the polynomial $p(t)$, and the leading term of the polynomial on the right is $$\sigma_1(b_m) \cdot \sigma_2(b_m) \cdot \ldots \cdot \sigma_{n}(b_m)\cdot t^{mn +1},$$ where $b_m$ is the leading coefficient of the polynomial $q(t)$. These terms should be equal. But~$kn \neq mn + 1$, because~$n > 1$. We may assume that~${mn + 1 > kn}$ without loss of generality. Since~$t$ is transcendental, if for some polynomial~$f$ one has~$f(t) = 0$, then~$f(x) \equiv 0$, so it means that $$\Nm_{L(t)/\Bbbk(t)}(b_m) = \sigma_1(b_m) \cdot \sigma_2(b_m)\cdot\ldots \cdot \sigma_{n}(b_m) = 0.$$ This is a contradiction, because~$\Nm_{L/\Bbbk}: L^* \to \Bbbk^*$ and~$b_m \in L^*$.
\end{proof}

\section{Balanced semidirect products}
\label{balance}

Now, let us define a balanced semidirect product. Let~$n$ be a positive integer. Write $$n = \prod\limits^m_{i=1}p^{r_i}_i,$$ where~$p_i$ are pairwise different prime numbers. One has canonical isomorphism $$\mumu_n \cong \prod\limits^{m}_{i=1}\mumu_n(p_i),$$ where $$\mumu_{n}(p_i) \cong \mumu_{p_i^{r_i}}.$$ Thereby, for the multiplicative group~$\mumu^*_n $ of residues modulo~$n$ that are coprime to~$n$, one has a canonical isomorphism $$\Aut(\mumu_n) = \mumu^*_n \cong \prod\limits^{m}_{i=1}\mumu_n(p_i)^*,$$ where $$\mumu_{n}(p_i)^* \cong \mumu^*_{p_i^{r_i}} \cong \mumu_{p_i - 1} \times \mumu_{p_i^{r_i - 1}}.$$
\begin{defn}
    \label{bal}
    Let~$q$ be a prime number. Suppose that~$n$ is divisible only by primes~$p_i$ congruent to 1 modulo~$q$, and let~$\chi: \mumu_q \to \mumu^*_n $ be a homomorphism. We say that~$\chi$ is {\sf balanced} if its composition with each of the projections~$\mumu_n^* \to \mumu_n(p_i)^*$ is an embedding. We say that a semidirect product~$G$ of~$\mumu_n$ and~$\mumu_q$ corresponding to the homomorphism~$\chi$ is {\sf balanced} if~$\chi$ is balanced.
\end{defn}

The following lemma shows that Definition \ref{bal} and definition which was given in the introduction are equivalent.  

\begin{lm}
 Let~$q$ be a prime number. Suppose that~$n$ is divisible only by primes~$p_i$ congruent to 1 modulo~$q$. Then $\mumu_n \rtimes_{\chi}\mumu_q$ is balanced if and only if $\mumu_n \rtimes_{\chi}\mumu_q$ has trivial center.
\end{lm}

\begin{proof}
Let $x$ be a generator of the group $\mumu_n$ and $y$ be a generator of the $\mumu_q$. Let $r \in \mumu_n^*$ such that $$y^{-1}xy = x^{r},$$ i.e. $\chi(y) = r$. Therefore $$y^{-1}x^my= x^{rm}$$ for $1 \leqslant m \leqslant n - 1$. 

Suppose that $\chi$ is balanced. Suppose that $y^bx^a$ is lying in the center for some~${0 \leqslant b \leqslant q - 1}$ and~${0 \leqslant a \leqslant n - 1}$. Then $$y^{b+1}x^a = y^bx^ay = y^{b+1}x^{ra}$$ and $$y^bx^{a+1} = xy^bx^a = y^bx^{r^b+a}.$$ Therefore $a(r - 1) \equiv 0(\text{mod }n)$ and $r^b \equiv 1(\text{mod } n)$. Since $\chi$ is balanced $r - 1$ is not divisible by any $p$, which divide $n$. Therefore $a = 0$. Note that $r^q \equiv 1 (\text{mod }n)$, since~$r$ is an image of the generator of order $q$. Therefore $r^{\gcd(b,q)} \equiv 1(\text{mod } n)$, but $\gcd(b,q)$ is either $1$ and in this case $\chi$ is not balanced or $q$ and in this case $b = 0$. The element~${y^bx^a = y^0x^0 = 1}$ is trivial and therefore the center is trivial.


Suppose that $\mumu_n^* \to \mumu_n(p)^*$ is not an embedding, then $x^{\frac{n}{p}}$ lies in the center of $\mumu_n \rtimes_{\chi}\mumu_q$. Indeed,~$x^{\frac{n}{p}}$ commutes with $x$, so it is enough to show that it is commutes with $y$. One has that $$y^{-1}x^{\frac{n}{p}}y = x^{r\frac{n}{p}}.$$ But $r - 1$ is divisible by $p$, since $\mumu_n^* \to \mumu^*_n(p)$ is not an embedding. Therefore $\frac{n}{p} (r - 1)$ is divisible by $n$, and so $x^{\frac{n}{p}} = x^{r\frac{n}{p}}$. Therefore $x^{\frac{n}{p}}$ and $y$ commute. Then center of the group~$\mumu_n \rtimes_{\chi}\mumu_q$ is non-trivial.

\end{proof}

The following lemma answers the question of when balanced semidirect products are isomorphic.  

\begin{lm}
\label{isomorphism}
  Let~$q$ be a prime number. Suppose that~$n$ is divisible only by primes~$p_i$ congruent to 1 modulo~$q$. Let ~$\chi_1,\chi_2: \mumu_q \to \mumu^*_n $ be two balanced homomorphisms. Denote~${G_i = \mumu_n \rtimes_{\chi_i} \mumu_q}$. Then $G_1 \cong G_2$ if and only if $\im(\chi_1) = \im(\chi_2)$.
\end{lm}

\begin{proof}
By \cite[Theorem 3.3]{semi} one has that  $G_1 \cong G_2$ if and only if there exists $\omega \in \Aut(\mumu_n)$ and $\nu \in \Aut(\mumu_q)$ such that $$\chi_1(\nu(x)) = \omega^{-1}\chi_2(x)\omega$$ for all $x \in \mumu_q$. Since $\Aut(\mumu_n) \simeq \mumu^{*}_n$ is an abelian group, therefore the latter condition can be rewritten as follows: $G_1 \cong G_2$ if and only if there exists $\nu \in \Aut(\mumu_q)$ such that~${\chi_1(\nu(x)) = \chi_2(x)}$ for all $x \in \mumu_q$. Since the image of $\chi_i$ is a cyclic group of order~$q$ one has that~${\chi_1(\nu(x)) = \chi_2(x)}$ if and only if~${\im(\chi_1) = \im(\chi_2)}$.
\end{proof}

Let us provide an example of two balanced semidirect products between groups $\mumu_n$ and~$\mumu_q$ for the fixed pair $(n,q)$, which are not isomorphic. The example is taken from~\cite[~Example 2.5]{Sh20}.

\begin{ex}
Let $q \geqslant 3$ be a prime number. Let $p_1$ and $p_2$ be distinct prime numbers congruent to $1$ modulo $q$,
and let $n=p_1p_2$. Then $\mumu_n^*\cong  \mumu_{p_1-1}\times\mumu_{p_2-1}$,
and each of the cyclic groups $\mumu_n(p_i)^*\cong\mumu_{p_i-1}$
contains a unique subgroup of order~$q$. Let $\delta_i$, $i=1,2$, be generators of these subgroups.
Let $d_1=\delta_1\delta_2$ and $d_2=\delta_1\delta_2^{-1}$.
Let $\chi_1, \chi_2\colon\mumu_q\to\mumu_n^*$ be the homomorphisms that send
a generator of $\mumu_q$ to $d_1$ and~$d_2$, respectively. Construct the groups
$G_1$ and $G_2$ as semidirect products of $\mumu_n$ and~$\mumu_q$ corresponding to the homomorphisms $\chi_1$ and $\chi_2$, respectively. Then $G_1$ is not isomorphic to $G_2$ by Lemma \ref{isomorphism} since~${\im(\chi_1) \neq \im(\chi_2)}$. 
\end{ex}

Let us give a motivation for the definitions that appear below. Let $X$ be a Severi--Brauer variety such that $\dim(X) = q - 1$ for some prime number $q \geqslant 3$.  Suppose~$\mumu_{n_1} \rtimes_{\chi_1} \mumu_q$ and~$\mumu_{n_2} \rtimes_{\chi_2} \mumu_q$ both act on $X$. Let $t = \gcd(n_1,n_2)$. Since $\mumu_{n_i}$ is cyclic for~$1 \leqslant i \leqslant 2$, then~$\mumu_t$ is a normal subgroup in $\mumu_{n_i} \rtimes_{\chi_i} \mumu_q$ for $1 \leqslant i \leqslant 2$. Therefore, one can restrict the action of $\mumu_q$ on $\mumu_{n_i}$ to the action on~$\mumu_t$. Let us denote the morphism from~$\mumu_q$ to $\mumu_t^*$, which is provided by the action, by $\chi_{i,t}$ for~$1 \leqslant i \leqslant 2$. Then $$\mumu_{t} \rtimes_{\chi_{i,t}} \mumu_q \subset \mumu_{n_i} \rtimes_{\chi_i} \mumu_q$$ acts on $X$ for~$1 \leqslant i \leqslant 2$. In Theorem \ref{consistancy} we will show that in this case~${\im(\chi_{1,t}) = \im(\chi_{2,t})}$. We want to provide definitions of collections of subgroups to be able to say that all groups which act on the variety have  ``the same type'' in some sense, which will be described below.

Let us fix a prime number $q \geqslant 3$. Let $\frak{I}_q$ be the set of all integers that are divisible only by primes congruent to~$1$ modulo~$q$. Let $\frak{X}$ be the set of all balanced maps~${\mumu_q \to \mumu^*_n}$, where $n \in \frak{I}_q$. 

Suppose that $n \in \frak{I}_q$ such that $n = \prod^m_{i=1}p^{r_i}_i$, where $p_i$ are prime numbers and $r_i$ are arbitrary integers, and $\chi : \mumu_q \to \mumu^*_n$. Then $p_i \neq q$ for all $1 \leqslant i \leqslant m$ and therefore $$\im(\chi) \cap \prod^m_{i = 1}\mumu_{p^{r_i - 1}_i} = \{1\}.$$ This shows that the image of the map~${\chi: \mumu_q \to \mumu^*_n}$ is contained in the subgroup~${\prod^m_{i = 1}\mumu_{p_i - 1}}$. 

Suppose $n$ is divisible by $t$, then $t = \prod^m_{i=1}p^{\alpha_i}_i$, where~${\alpha_i \leqslant r_i}$ for all $1 \leqslant i \leqslant m$. Denote $$\pi_{n,t} : \; \prod^m_{i = 1}\mumu_{p_i - 1} \longrightarrow \prod_{i \mid \alpha_i \neq 0}\mumu_{p_i - 1}.$$
Let $\Delta_n \subset \mumu_n^*$, then we denote $$\Delta_n|_{\mumu^*_t} = \pi_{n,t}(\chi(\Delta_n)).$$


\begin{defn}
Let $n_1,n_2$ be some integers such that $n_1,n_2 \in \frak{I}_q$. Denote $t = \gcd(n_1,n_2)$. Suppose $\Delta_{n_1}$ and $\Delta_{n_2}$ are subgroups of $\mumu^*_{n_1}$ and $\mumu^*_{n_2}$ respectively. A pair of subgroups~$(\Delta_{n_1},\Delta_{n_2})$ is called~{\sf consistent} if $$\Delta_{n_1}|_{\mumu^*_t} = \Delta_{n_2}|_{\mumu^*_t}.$$
\end{defn}

Notice that for any prime number $p \in \frak{I}_q$ and any positive integer $r$ there is only one subgroup $\Delta_{p^r} \subset \mumu^*_{p^r}$ such that the group $\Delta_{p^r}$ has order~$q$. Therefore all pairs of subgroups~$(\Delta_{p^rm},\Delta_{p^r})$ are consistent, for all $m \in \frak{I}_q$.



\begin{defn}
\label{hastype}

Let $\frak{K} \subset \{(n,\chi) \in \frak{I}_q \times \frak{X} \mid\chi : \; \mumu_q \to \mumu^*_n\}$.  Suppose that $\frak{A}$ is a collection of groups $$\{\mumu_n \rtimes_{\chi} \mumu_q\}_{(n, \chi) \in \frak{K}}.$$ The collection $\frak{A}$ is {\sf a consistent balanced collection} if all pairs of subgroups $(\im(\chi_1),\im(\chi_2))$ are consistent for all $\chi_1,\chi_2 \in \pi(\frak{K})$, where $\pi: \frak{I}_q \times \frak{X} \to \frak{X}$ is the natural projection. In particular, for $(n,\chi_1),(n,\chi_2) \in \frak{K}$ one automatically has that~${\im(\chi_1) = \im(\chi_2)}$.

\end{defn}


The following definition is a generalization of the previous definition.

\begin{defn}
\label{hastypelong}
Let $\frak{K} \subset \{(n,\chi) \in \frak{I}_q \times \frak{X} \mid\chi : \; \mumu_q \to \mumu^*_n \}$. Suppose that $\frak{B}$ is a collection of groups $$\{\mumu_q \times (\mumu_n \rtimes_{\chi} \mumu_q)\}_{(n, \chi) \in \frak{K}}.$$ The collection $\frak{B}$ is {\sf a consistent balanced collection} if all pairs of subgroups $(\im(\chi_1),\im(\chi_2))$ are consistent for all $\chi_1,\chi_2 \in \pi(\frak{K})$, where $\pi: \frak{I}_q \times \frak{X} \to \frak{X}$ is the natural projection. In particular, for $(n,\chi_1),(n,\chi_2) \in \frak{K}$ one automatically has that~${\im(\chi_1) = \im(\chi_2)}$.
\end{defn}


The following lemma shows when groups of type $\mumu_q \times (\mumu_n \rtimes_{\chi} \mumu_q)$ are isomorphic.

\begin{lm}
\label{correctness}
 Let $q$ be a prime number. Suppose that $G_i = \mumu_q \times (\mumu_n \rtimes_{\chi_i} \mumu_q)$ is a group such that $\mumu_n \rtimes_{\chi_i} \mumu_q$ is a balanced semidirect product for $1 \leqslant i \leqslant 2$. Then $G_1 \cong G_2$ if and only if $\im(\chi_1) = \im(\chi_2)$.
\end{lm}

\begin{proof}
If $\im(\chi_1) = \im(\chi_2)$ then by Lemma \ref{isomorphism} one has an isomorphism $G_1 \cong G_2$. Suppose~$G_1 \cong G_2$, since both $\chi_1$ and $\chi_2$ are balanced then $Z(G_1) \cong Z(G_2)$ and both are $\mumu_q$, where~$Z(G)$ is the center of a group $G$. Then $$(\mumu_n \rtimes_{\chi_1} \mumu_q) = G_1 / Z(G_1) \cong G_2/Z(G_2) = (\mumu_n \rtimes_{\chi_2} \mumu_q).$$ Therefore by Lemma \ref{isomorphism} one has $\im(\chi_1) = \im(\chi_2)$.

\end{proof}






The following lemma shows that if for a group $G$ one has that~${G \cong \mumu_q \times (\mumu_n \rtimes_{\chi_i} \mumu_q)}$ for some prime $q \geqslant 3$ and $\mumu_n \rtimes_{\chi_i} \mumu_q$ is a balanced semidirect product, then such a decomposition of the group is unique. In particular, one can talk about the group~${G \cong \mumu_q \times (\mumu_n \rtimes_{\chi_i} \mumu_q)}$ as a balanced group, and this definition does not depend on the decomposition $G$.

\begin{lm}
\label{uniq}
Let $q_1,q_2 \geqslant 3$ be prime numbers. Let $n \in \frak{I}_{q_1}$ and $m \in \frak{I}_{q_2}$. Let $\mumu_n \rtimes_{\chi_n} \mumu_{q_1}$ and $\mumu_m \rtimes_{\chi_{m}} \mumu_{q_2}$ be balanced semidirect products. Then $$\mumu_{q_1} \times (\mumu_n \rtimes_{\chi_n} \mumu_{q_1}) \cong \mumu_{q_2} \times (\mumu_m \rtimes_{\chi_m} \mumu_{q_2})$$ if and only if $q_1 = q_2$, $n = m$ and $\im(\chi_n) = \im(\chi_m)$
\end{lm}

\begin{proof}
If $q_1 = q_2$, $n = m$ and $\im(\chi_n) = \im(\chi_m)$ then by Lemma \ref{isomorphism} one has an isomorphism. Suppose $$\mumu_{q_1} \times (\mumu_n \rtimes_{\chi_n} \mumu_{q_1}) \cong \mumu_{q_2} \times (\mumu_m \rtimes_{\chi_m} \mumu_{q_2}).$$ Then since semidirect products are balanced one has that $$\mumu_q = Z(\mumu_{q_1} \times (\mumu_n \rtimes_{\chi_n} \mumu_{q_1})) \cong Z(\mumu_{q_2} \times (\mumu_m \rtimes_{\chi_m} \mumu_{q_2})) = \mumu_{q_2}.$$ Therefore, orders of the centers are the same and this means that $q_1 = q_2$. Since $q_1 = q_2$ and the orders of the groups are the same, then $n = m$. The rest directly follows from Lemma~\ref{correctness}.
\end{proof}

Let us show that two definitions of a consistent balanced collection agree with each other.

\begin{lm}
Let $\frak{K} \subset \{(n,\chi) \in \frak{I}_q \times \frak{X} \mid\chi : \; \mumu_q \to \mumu^*_n \}$. Suppose that $\frak{B}$ is a collection of groups $$\{\mumu_q \times (\mumu_n \rtimes_{\chi} \mumu_q)\}_{(n, \chi) \in \frak{K}}.$$ Then $\frak{B}$ is a consistent balanced collection if and only if a collection $$\frak{A} = \{\mumu_n \rtimes_{\chi} \mumu_q \mid \mumu_q \times (\mumu_n \rtimes_{\chi} \mumu_q )\in \frak{B}\}$$ is a consistent balanced collection.
\end{lm}

 \begin{proof}
 By Lemma \ref{uniq} for any $G \in \frak{B}$ there is a unique decomposition of $G$ as~${\mumu_q \times (\mumu_n \rtimes_{\chi} \mumu_q)}$. For groups in $\frak{A}$, one also has that the decomposition as $\mumu_n \rtimes_{\chi} \mumu_q$ is unique. Therefore, definitions agrees and $\frak{B}$ is a consistent balanced collection if and only if $\frak{A}$ is a consistent balanced collection.
 \end{proof}

\begin{thm}
\label{consistancy}
  Let $q$ be a prime number. Let $X$ be a non-trivial Severi--Brauer variety over a field $\mathbb{K}$, where $\dim(X) = q - 1$. Suppose $\frak{A} = \{\mumu_n \rtimes_{\chi} \mumu_q \mid  \mumu_n \rtimes_{\chi} \mumu_q \subset \Aut(X) \}$. Then~$\frak{A}$ is a consistent balanced collection. In particular, if both $\mumu_n \rtimes_{\chi_1} \mumu_q$ and $\mumu_n \rtimes_{\chi_2} \mumu_q$ act on~$X$ then $\im(\chi_1) = \im(\chi_2)$.
\end{thm}

\begin{proof}
Let us firstly show that if for some~$t \in \frak{I}_q$ there are two subgroups $$G_i \cong \mumu_t \rtimes_{\chi_i} \mumu_q \subset \Aut(X)$$ for $1 \leqslant i \leqslant 2$, then $\im(\chi_1) = \im(\chi_2)$. Suppose there exists~$t \in \frak{I}_q$ such that there are $$G_i \cong \mumu_t \rtimes_{\chi_i} \mumu_q \subset \Aut(X)$$ with $\im(\chi_1) \neq \im(\chi_2)$. Let $\theta_1$ and $\theta_2$ be generators of $\mumu_t$ in $G_1$ and $G_2$ respectively. By~Lemma~\ref{Aut} one has that~${\Aut(X)\cong A^*/ \mathbb{K}^*}$, so $\theta_1$,$\theta_2 \in A^*/ \mathbb{K}^*$. By Lemma~\ref{root} one can choose~${u_1}$,${u_2 \in A^*}$ which are lifts of $\theta_1$,$\theta_2,$ such that $u^t_1 = u^t_2 = 1$. Then~$\mathbb{K}(u_1)$ and~$\mathbb{K}(u_2)$ are isomorphic extensions of $\mathbb{K}$, since they both are extension of $\mathbb{K}$ by roots of unity of degree $t$. Let us find the group $\Gal(\mathbb{K}(u_1)/\mathbb{K})$. Since $A$ is also $\mathbb{K}(u_1)$-algebra one has $$1 < \left[\mathbb{K}(u_1):\mathbb{K}\right]  = \frac{\left[A:\mathbb{K}\right]}{\left[A:\mathbb{K}(u_1)\right]} = \frac{q^2}{\left[A:\mathbb{K}(u_1)\right]}.$$ One has $A \neq \mathbb{K}(u_1)$ since $A$ is a central algebra, then $\left[\mathbb{K}(u_1):\mathbb{K}\right] = q$ and so~$\Gal(\mathbb{K}(u_1)/\mathbb{K}) = \mumu_q$. Denote an isomorphism between $\mathbb{K}(u_1)$ and $\mathbb{K}(u_2)$ by $f$ and the inclusion of $\mathbb{K}(u_i)$ into~$A$ by~$\rho_i$. Then by Skolem–Noether theorem there exists $a \in A^{*}$ such that $\rho_1(x) = a^{-1}\rho_2(f(x))a$ for all~$x \in \mathbb{K}(u_1)$. Let us replace the group $G_2$ by a group~$a^{-1}G_2a$ which is isomorphic to~$G_2$. Then one can choose $u_1 = u_2$, since they both are roots of unity in the same field extension. Since the generator $\alpha_1$ of $\mumu_q$ acts on $\theta_1$ by sending $\theta_1$ to some power of~$\theta_1$, we can extend this action of $\alpha_1$ to the whole field~$\mathbb{K}(u_1)$ such that $\alpha_1$ acts trivially on~$\mathbb{K}$. Denote this extension of the action by $\widetilde{\alpha_1}$. By the construction~$ \widetilde{\alpha_1} \in \Gal(\mathbb{K}(u_1)/\mathbb{K})$ and therefore~${\widetilde{\mumu_q} \subset \Gal(\mathbb{K}(u_1)/\mathbb{K})}$. Note that the action of $\widetilde{\mumu_q}$ on $u$ is the same as the action $\mumu_q$ on $\mumu_n$ then we have that $\chi_1(\mumu_q) \subset \Gal(\mathbb{K}(u_1)/\mathbb{K})$. By the same argument~${\chi_2(\mumu_q) \subset \Gal(\mathbb{K}(u_1)/\mathbb{K})}$. Also, the order $$|\im(\chi_1)|= |\im(\chi_2)| =  |\Gal(\mathbb{K}(u_1)/\mathbb{K})| = q.$$ Therefore $$\im(\chi_1) = \Gal(\mathbb{K}(u_1)/\mathbb{K}) = \im(\chi_2),$$ a contradiction. 
  

Let us show that $\frak{A}$ is a consistent balanced collection. Suppose that there is a pair of subgroups $\mumu_{n_1} \rtimes_{\chi_{n_1}} \mumu_q$ and $\mumu_{n_2} \rtimes_{\chi_{n_2}} \mumu_q$ of $\Aut(X)$ which provide a non consistent pair~${(\im(\chi_{n_1}),\im(\chi_{n_2}))}$. Let~${t = \gcd(n_1,n_2)}$. Then there are~${\chi_{n_{1,t}},\chi_{n_{2,t}}: \mumu_q \to \mumu^*_t}$, morphisms corresponded to the restrictions of the action $\mumu_q$ to $\mumu_t$, such that $$\im(\chi_{n_1})|_{\mumu^*_t}  = \im(\chi_{n_{1,t}}) \neq \im(\chi_{n_{2,t}}) = \im(\chi_{n_2})|_{\mumu^*_t}.$$ Notice that both groups $$G_i \cong \mumu_t \rtimes_{\chi_{i,t}} \mumu_q$$ are isomorphic to a subgroup of $\Aut(X)$ for $i = 1$, $2$. This gives a contradiction with what was proven above, since $\im(\chi_{n_{1,t}}) \neq \im(\chi_{n_{2,t}})$.

\end{proof}

\section{Example in characteristic 0}
\label{char0}
In this section, we prove Theorem \ref{example}. The idea of the construction is as follows. By Lemma \ref{sur}, one has a construction of a non-trivial Severi--Brauer variety of dimension~$n-1$ by any cyclic extension~$L/\Bbbk$ of degree~$n$. We simply consider the cyclic algebra~$(L(t)/\Bbbk(t), \sigma,t)$, by Lemmas \ref{sur} and  \ref{nm}, this algebra does not split, so the corresponding Severi--Brauer variety is non-trivial.

\begin{proof}[Proof of Theorem \ref{example}]

    Let~$\omega$ and~$\tau$ be non-trivial~$q$-th root of unity and~$q$-th root of~$2$, respectively. For any prime number~$p$ and for any positive integer~$r$, let us denote a primitive root of unity of degree~$p^r$ by~$\xi_{p^r}$. Let $D = \{n \in \mathbb{Z} \mid \mumu_q \times (\mumu_n \rtimes_{\chi}\mumu_q) \in \frak{A} \}$ for some balanced map $\chi$. 
    
    Let us take field extension~$L$ of~$\mathbb{Q}(\omega)$ generated by~$\tau$ and all~$\xi_{p^r}$, where $p^r$ is a divisor of some $n \in D$. More precisely, let us enumerate all pairs~$(p,r)$, where~$p$ is a prime number such that~$p \equiv 1 \text{ (mod }q$) and~$r$ is a positive integer, such that $p^r$ is a divisor of some $n \in D$, by positive integers (if the set $D$ is finite, then enumeration is by first several integers). Let~$(p,r)$ corresponds to~$t$. Then let us denote~$\zeta_t = \xi_{p^r}$. Then~$L = \mathbb{Q}(\omega, \tau, \zeta_1, \zeta_2, \zeta_3,\ldots)$. If~$D$ was a finite set, then $L$ is a finite extension of $\mathbb{Q}(\omega)$, otherwise it is an infinite extension. Let us denote $$B_j = \{p \mid \exists r \text{ and } t \in \mathbb{N} \text{ such that } t \leqslant j \text{ and } t \text{ corresponds to } (p,r) \}$$ in other words it is a set of all prime numbers which appears in the first $j$ pairs. And $$C_{j,p} = \{r\mid \exists t \in \mathbb{N} \text{ such that } t \leqslant j \text { and } t \text{ corresponds to } (p,r) \},$$ in other words it is a set of powers in which $p$ appears in the first $j$ pairs. Let $r_j(p)$ be the maximum of the set~$C_{j,p}$. Let us denote $L_j = \mathbb{Q}(\omega, \tau, \zeta_1, \zeta_2, \ldots ,\zeta_j)$. Notice that $$L_j =  \mathbb{Q}(\omega, \tau, \xi_{N(j)}),$$ where~$\xi_{N(j)}$ is root of unity of degree~$N(j) = \prod_{p \in B_j}p^{r_j({p)}}$.
    
    
    By construction, the Galois group of the extension~$L/\mathbb{Q}(\omega)$ is a profinite group, which is realized as an inverse limit of Galois groups of extensions~$L_j/\mathbb{Q}(\omega)$. One can notice that $$ \Gal(L_j/\mathbb{Q}(\omega)) = \Gal(\mathbb{Q}(\omega, \tau, \xi_{N(j)})/\mathbb{Q}(\omega)) \simeq \mumu_q \times \mumu^*_{N(j)} \simeq \mumu_q \times \prod\limits_{p_i \in B_j}\mumu^*_{N(j)}(p_i).$$ 
    Let us denote 
    \begin{equation}
        x_{j,p}= \frac{(p -1)p^{r_j(p) -1}}{q}.
        \label{form}
    \end{equation} 
    Let $\chi_{N(j)}: \mumu_q \to \mumu^*_{N(j)}$ be a map which agrees with all maps $\chi: \mumu_q \to \mumu^*_n$ appears from groups $\mumu_q \times (\mumu_n \rtimes_{\chi} \mumu_q) \in \frak{A}$, i.e. if $N(j)$ was divisible by $n$ then $\chi_{N(j)} = \pi_{n,N(j)} \circ \chi$, where $\pi_{n,N(j)}$ is the natural projection $\mumu_{N(j)}^*$ to $\mumu_n^*$. Since the collection $\frak{A}$ is consistent, this map exists. Consider the element $$\sigma_j = \gamma \times \prod\limits_{p \in B_j} \epsilon^{x_{j,p}}_{j,p},$$ of $\mumu_q \times \prod\limits_{p \in B_j}\mumu^*_{N(j)}(p)$, where~$\epsilon_{j,p}$ is a generator of~$\mumu^*_{p^{r_j(p)}}$ and~$\gamma$ is a generator of~$\mumu_q$, such that $\im(\chi_{N(j)})$ is a group generated by $\sigma_j$. Denote $\im(\chi_{N(j)})$ by~$\Delta_j$.
    
    
    Let us denote the group of~$p$-adic integers by~$\mathbb{Z}_p$, and the multiplicative group of invertible $p$-adic integers by $\mathbb{Z}^*_p$. One has that
    \begin{multline*}
        \Gal (L/\mathbb{Q}(\omega))= \varprojlim\Gal(L_j/\mathbb{Q}(\omega))= \\
        = \varprojlim (\mumu_q \times \prod\limits_{p \in B_j} \mumu^*_{N(j)}(p))= \mumu_q \times \prod_{\substack{p\in D\\ p \text{ prime}}} \mathbb{Z}_{p}^* \simeq \mumu_q \times \prod \limits_{\substack{p\in D\\ p \text{ prime}}} \mumu_{p - 1} \times \prod \limits_{\substack{p\in D\\ p \text{ prime}}} \mathbb{Z}_{p},
    \end{multline*}
    because there is an isomorphism~$\mathbb{Z}_p^* \simeq \mumu_{p-1} \times \mathbb{Z}_p$. 
    
    Consider the element $$\sigma = \gamma \times \prod_{\substack{p\in D\\ p \text{ prime}}} \epsilon^{\frac{p -1}{q}}_p \times \prod_{\substack{p\in D\\ p \text{ prime}}} 0 $$ of $\mumu_q \times \prod \limits_{\substack{p\in D\\ p \text{ prime}}} \mumu_{p - 1} \times \prod \limits_{\substack{p\in D\\ p \text{ prime}}} \mathbb{Z}_{p},$ where~$\epsilon_p$ is a generator of~$\mumu_{p -1}$ and~$\gamma$ is a generator of~$\mumu_q$, such that $\sigma|_{L_j} = \sigma_j$. Since collection $\frak{A}$ was consistent one can choose such element $\sigma$. Let $\Delta$ be a group generated by $\sigma$. Notice that~$\Delta \simeq \mumu_q$ is a subgroup of order~$q$ of $$\Gal(L/\mathbb{Q}(\omega)) = \mumu_q \times \prod \limits_{\substack{p\in D\\ p \text{ prime}}} \mumu_{p - 1} \times \prod \limits_{\substack{p\in D\\ p \text{ prime}}} \mathbb{Z}_{p}.$$ Since~$\Delta$ is a finite subgroup, one can notice that it is a closed subgroup of~$\Gal(L/\mathbb{Q}(\omega))$.  
    
    It means that we can take~$\Bbbk = L^{\Delta}$, so that $L/ \Bbbk$ is a Galois extension of degree~$q$. Additionally one has $$\Gal(L/\Bbbk) = \Delta \simeq  \mumu_q.$$ 
    
    Notice that~$L_j^{\Delta_j} = \Bbbk\cap L_j$ and so $$\left[L_j : \Bbbk\cap L_j\right] = \left[L_j : L_j^{\Delta_j}\right] = q$$ and finally~$\Gal(L_j / \Bbbk\cap L_j)$ is generated by $\sigma|_{L_j} = \sigma_j$. So if~$l \in L_j$ then by the construction one has~$\sigma(l) = \sigma|_{L_j}(l) = \sigma_j(l)$.
    
    One can construct by the field extension $L(t)/\Bbbk(t)$ a non-trivial Severi--Brauer variety~$X$. Let~${A=(L(t)/\Bbbk(t),\sigma,t)}$ be a corresponding to the variety central simple algebra. Then by Lemma \ref{Aut} one has~${\Aut(X) = A^*/\Bbbk(t)^*}$. We claim that $X$ is the variety we want, i.e., such a variety that any group $G \in \frak{A}$ is acting on $X$.

    Fix a positive integer $m$ and some non-negative integers $r_1,r_2,\ldots,r_m$. Consider $m$ different prime numbers~$p_1,\ldots,p_m$, such that $p_i \equiv 1(\text{mod }q)$ for $1 \leqslant i \leqslant m$. Let us denote~${n = \prod^m_{i=1}p_i^{r_i}}$. Suppose that $n \in D$.
    
    Consider the element $$\theta = \xi_{p^{r_1}_{{1}}} \cdot \ldots\cdot \xi_{p^{r_m}_{{m}}}.$$ Since $n \in D$, then $\theta \in L$. Let~$j$ be the minimal number such that~$\theta \in L_j$. Notice that~${\theta^n = 1 \in \Bbbk(t)}$. Suppose that~$\theta^x \in \Bbbk(t)$ for some positive $x < n$, then by the construction $\theta^x \in \Bbbk$. So $\theta^x$ is fixed by~$\sigma_j$. Note that $$\sigma_j(\xi_{p^{r_i}_i}) = \xi_{p^{r_i}_i}^{\epsilon^{x_{j,p_i}}_{j,p_i}}$$ for $1 \leqslant i \leqslant m$,  therefore $$\sigma_j(\theta^x) = \sigma_j(\xi^x_{p^{r_1}_1}) \cdot \ldots \cdot \sigma_j(\xi^x_{p_m^{r^m}}) = \xi^{x\epsilon^{x_{j,p_1}}_{j,p_1}}_{p_1^{r_1}} \cdot \ldots \cdot \xi^{x\epsilon^{x_{j,p_m}}_{j,p_m}}_{p_m^{r_m}} = \xi_{p_1^{r_1}}^{x} \cdot \ldots \cdot \xi^{x}_{p_m^{r_m}} = \theta^x.$$ Then one has equality $$\xi^{x\left(\epsilon^{x_{j,p_1}}_{j,p_1} - 1\right)}_{p_1^{r_1}} \cdot \ldots \cdot \xi_{p_m^{r_m}}^{x\left(\epsilon^{x_{j,p_m}}_{j,p_m} - 1\right)} = 1.$$ Since all $\xi_{p^{r_i}_i}$ are primitive roots of unity of degree $p_i^{r_i}$ for different $p_i$ for $1 \leqslant i \leqslant m$, the product of its powers is equal $1$ if and only if all powers are divided by corresponding $p^{r_h}_h$. Since $x < n$ there is at least one $p_h$ such that $v_{p_h}(n) > v_{p_h}(x)$, let us consider such $p_h$. Then from $$x\left(\epsilon^{x_{j,p_h}}_{j,p_h} - 1\right) \text{ is divisible by }p^{r_h}_h$$ we have that $\epsilon^{x_{j,p_h}}_{j,p_h} - 1 \text{ is divisible by }p_h.$ Since $\epsilon_{j,p_h}$ is a a generator of the group $\mumu^*_{p_h^{r_j(p_h)}}$, then $\epsilon_{j,p_h}$ is a generator of a group $\mumu^*_{p_h}$ after considering all remainders modulo $p_h^{r_j(p_h)}$ as remainders modulo $p_h$. This means that for any integer $y$, one has that $\epsilon_{j,p_h}^y - 1$ is divisible by $p_h$ if and only if $y$ is divisible by $p_h$. But formula \ref{form} shows that $x_{j,p_h}$ is never divisible by $p_h$ since $q \neq p_h$.

    So~$\theta$ has order exactly~$n$.

Let $G_n \subset A^*/\Bbbk(t)^*$ be the subgroup generated by images of~$\theta$ and~$\alpha$, where $\alpha$ is the standard generator of the cyclic algebra $A$ (so $\alpha^q = t$). Let us prove that $G_n$ is isomorphic to~${\mumu_n \rtimes_{\chi} \mumu_q}$, where the semidirect product is balanced and $\mumu_q \times (\mumu_n \rtimes_{\chi} \mumu_q) \in \frak{A}$. Since~$\alpha^i \notin \Bbbk(t)$ for all~$1 \leqslant i \leqslant q-1$ and~$\theta^x \notin \Bbbk(t)$ for~${1 \leqslant x \leqslant n-1}$ (as was proved above), the images of $\alpha$ and $\theta$ in $A^{*}(\Bbbk(t))/\Bbbk(t)^{*}$ have orders $q$ and $n$, respectively. By definition one has $$\theta \alpha = \alpha \sigma(\theta) = \alpha \sigma|_{L_j}(\theta) = \alpha\sigma_j(\theta).$$ By the construction, $\sigma_j(\theta) = \chi(\theta).$ Therefore, $G_n$ is isomorphic to a semidirect product of $\mumu_n$ and $\mumu_q$ which is isomorphic exactly $\mumu_n \rtimes_{\chi} \mumu_q$. 


    
    Consider the group $G(n)$ generated by~$G_n$ and~$\tau$. The elements~$\theta$ and~$\tau$ commute because they both lie in the field $L(t)$. Notice that~$\tau \in L_0$, and so $$\sigma(\tau) = \sigma|_{L_0}(\tau) = \omega\tau.$$ Then $$\tau \alpha = \alpha \sigma(\tau) = \omega \alpha \tau,$$ so in~$A^*/\Bbbk(t)^*$ elements~$\alpha$ and~$\tau$ commute. This means that $G(n)$ is isomorphic to~${\mumu_q \times G_n \cong \mumu_q \times (\mumu_n \rtimes_{\chi} \mumu_q)}$.

The integer $m$ and the integers $r_1,\ldots,r_m$ are arbitrary. Hence, for all $n \in D$, there is a subgroup $G(n)$ of $\Aut(X)$ such that $G(n) \simeq \mumu_q \times (\mumu_n \rtimes_{\chi} \mumu_q).$


Consider the case where $q = 3$. Then $\omega$ is a non-trivial root of unity of degree~$3$. In this case, $X$ is a non-trivial Severi--Brauer surface. By \cite[Proposition 1.5(ii)]{Vicul} one has~${\mumu_3^3 \subset \Bir(X)}$, because $\omega \in \Bbbk(t)$ by the construction. So the surface $X$ has also $\mumu_3^3$ as a subgroup of $\Bir(X)$.
\end{proof}

\section{Classification in positive characteristic}
\label{criter}

In this section, we prove Theorems \ref{maincon} and \ref{mainconp}.  

Let $l$ and $n$ be integers such that $\gcd(n,l) = 1$. Then let~$\ord_n(l)$ denote the minimal integer~$r$ such that~$l^r - 1$ is divisible by~$n$. 

Let $q$ be a prime number and $t$ be some integer. Let~$v_q(t)$ be the maximal~$r$ such that~$t$ is divisible by~$q^r$.


Firstly, let us provide an improved version of the Theorem \ref{Sav}. The proof of the new version is exactly the same as that given in \cite[Section 7]{ASav} with small modifications. In the case where $X$ is a Severi--Brauer variety over a field of positive characteristic, we need a smaller class of consistent balanced collections. 

\begin{defn}
Let $l$ and $q$ be prime numbers such that $l \neq q$. Suppose that~$n$ is divisible only by primes congruent to $1$ modulo~$q$, and let~$\chi: \mumu_q \to \mumu^*_n $ be a balanced homomorphism. The homomorphism said to be {\sf of Frobenius type} if $\chi$ is as follow: $$\chi(y): x \mapsto x^{l^{\left(d_0 \cdot d \cdot q^{(k-1)}\right)}},$$ for $y$ the generator of $\mumu_q$ and $x \in \mumu_n$, where $\ord_n(l) = q^kd_0$, such that $d_0$ is not divisible by $q$, and $1 \leqslant d \leqslant q -1$. We say that a semidirect product of~$\mumu_n$ and~$\mumu_q$ corresponding to the homomorphism~$\chi$ is {\sf Frobenius} if~$\chi$ is of Frobenius type.
\end{defn}

Note that Frobenius semidirect product between $\mumu_n$ and $\mumu_q$ is unique with respect to isomorphism.

Let us specialize the previous definition to the families of balanced semidirect products, i.e. Definitions \ref{hastype} and \ref{hastypelong}. Let us fix a prime number $q \geqslant 3$, a prime number $l \neq q$, and a~positive integer $k$. Let~$\frak{I}_{k,q,l}$ be the set of all integers that are divisible only by primes~$p$ congruent to~$1$ modulo~$q$, such that ${v_q(\ord_{p}(l)) = k}$. 

\begin{defn}
  \label{Ftype}
 Let $\frak{J} \subset \frak{I}_{k,q,l}$ and let $\frak{D} = \{\mumu_n \rtimes_{\chi_n}\mumu_q \}_{n \in \frak{J}}$ (or $\frak{D} = \{\mumu_q \times (\mumu_n \rtimes_{\chi_n}\mumu_q) \}_{n \in \frak{J}})$ be a consistent balanced collection. The collection $\frak{D}$ is called {\sf Frobenius} if all semidirect products $\mumu_n \rtimes_{\chi_n}\mumu_q$ are Frobenius, i.e., all maps $\chi_n$ are of Frobenius type.

\end{defn}
 
 Let us finally give a modification of Theorem \ref{Sav}. Let us again mention that the proof of the new version is exactly the same as the one from \cite[Section 7]{ASav} with small remarks. The author does not claim to be original in this part of the text.  
 
\begin{thm}
\label{ASavnew}
Let~$X$ be a non-trivial Severi--Brauer variety of dimension~$q - 1$ over a field~$\mathbb{K}$, where~$q \geqslant 3$ is prime, $q \neq \Char(\mathbb{K})$ and $\Char(\mathbb{K}) = l > 0$. Let~$G$ be a finite subgroup of~$\Aut(X)$, then there exists a positive integer~$n$ and a map $\chi$ such that~$G$ is isomorphic to a subgroup of~$\mumu_q \times (\mumu_n \rtimes_{\chi} \mumu_q)$, where the semidirect product is balanced and Frobenius.
\end{thm}
\begin{proof}
  Suppose that $A$ is the central simple algebra which corresponds to the Severi--Brauer variety $X$. 
  Consider the subgroup $N_G \lhd G$ which include all the elements of $G$ that admit a lift of finite order in $A^*$ (for more detail see \cite[Corollary 2.2]{ASav}). In the case when $N_G$ is trivial \cite[Lemma 7.7]{ASav} yields the claim.

Suppose now that $N_G$ is non-trivial. Consider a non-trivial normal cyclic subgroup~${\Theta \lhd G}$ obtained by \cite[Proposition 7.5]{ASav}. Consider the subfield $L \subset A$ generated over~$\mathbb{K}$ by the preimage of $\Theta$ under the canonical projection. Since $\Theta$ is a normal subgroup of $G$, the group $G$ acts on $K$ by conjugations, which yields a map $\varphi \colon G \to \Aut(L/\mathbb{K})$. Denote by $G'$ its kernel. By \cite[Theorem 7.1]{ASav} every non-trivial field extension of $\mathbb{K}$ contained in $A$ is maximal, hence any lift of any element of $G'$ lies in $L^*$, and therefore $G' \subset L^*/\mathbb{K}^*$. We now apply \cite[Proposition 7.6]{ASav} to get that $G'$ is isomorphic either to $\mumu_n$ or to $\mumu_n \times \mumu_q$, where $n$ is coprime with $q$ and the unique subgroup of $G'$ isomorphic to $\mumu_n$ consists of all the elements of $G'$ that admit a lift of finite order in $L^*$.

Note that the group $\Aut(L/\mathbb{K})$ is either $\mumu_q$ or trivial. Consider the case when $n = 1$. The group $G$ then has the order dividing $q^2$, so to prove the assertion it remains to exclude the $\mumu_{q^2}$ case. Suppose that $G \cong \mumu_{q^2}$. As $n = 1$, no element of $G' \cong \mumu_{q} \subset \mumu_{q^2}$ admits a lift of finite order, therefore the group $N_G$ is trivial, which contradicts the assumption. This completes the proof in the case when~$n = 1$. 

Suppose now that $n \ne 1$. In this case, the extension $L \supset \mathbb{K}$ is cyclotomic, as it can be generated by the lift of the generator of $\mumu_n$ of finite order, thus it is Galois. Recall from~\cite[~Theorem 7.1]{ASav} that $L$ is of degree $q$ over~$\mathbb{K}$, hence its group of automorphisms~$\Gal(L/\mathbb{K})$ is isomorphic to~$\mumu_q$. Denote by $u \in K$ any preimage of finite order of the generator $x$ of $\mumu_n$. Note that $u$ is a root of unity. The Galois group $\Gal(L/\mathbb{K})$ then acts on the group generated by $u$ in such a way that the only elements preserved by this action are those contained in $\mathbb{K}^*$. Taking quotient by $\mathbb{K}^*$ we thus obtain that the action of $\Gal(L/\mathbb{K}) \cong \mumu_{q}$ on the group generated by $u$ induces an action on $\mumu_n$. 

More precisely morphism which is given by the action is~${\chi: \Gal(L/\mathbb{K}) \to \mumu_n^*}$, where~$\chi$ is a generator of~$\Gal(L/\mathbb{K})$. Since we can choose a lift~$u$ of a generator of $\mumu_n$ such that~${u^n = 1}$ then by Theorem \ref{NI} one has that $$\mumu_q \cong \Gal(L/\mathbb{K}) \cong \Gal(\mathbb{F}_{l}(u)/\mathbb{K} \cap \mathbb{F}_{l}(u))$$ and since $\mathbb{F}_{l}(u)$ is a finite extension of $\mathbb{F}_l$ a generator of the group $\Gal(\mathbb{F}_{l}(u)/\mathbb{K} \cap \mathbb{F}_{l}(u))$ is $$\chi(y): x \to x^{l^{\left(d_0 \cdot d \cdot q^{(k-1)}\right)}},$$ for $y$ the generator of $\mumu_q$ and $x \in \mumu_n$, where $\ord_n(l) = q^kd_0$, such that $d_0$ is not divisible by $q$, and $1 \leqslant d \leqslant q -1$. Note additionally that we got the fact that $v_q(\ord_n(l)) \geqslant 1$ and in particular that $n$ is not divisible by $l$. (This was the only part of the proof which is different from the original version from \cite[Section 7]{ASav}).

We can now conclude the proof considering the image of $G$ under $\varphi$ in $\Gal(L/\mathbb{K}) \cong \mumu_{q}$. If it is trivial, the proof is complete, so consider the case when $\im \varphi \cong \mumu_{q}$. Then $G$ fits into a short exact sequence
$$1 \to G' \to G \xrightarrow{\varphi} \mumu_q\to 1.$$
We proceed by proving that $\varphi$ admits a right inverse. Consider any preimage $\chi_1 \in G$ under $\varphi$ of the generator $\chi$ of the group $\Gal(L/\mathbb{K})$. By the discussion above, $\chi_1$ does not commute with any non-trivial element of $\mumu_n$, thus the order of $\chi_1$ is either $q$ or $q^2$. Due to \cite[Proposition 7.6]{ASav} no element of order $q$ can admit a lift of finite order in $A$, so the order of $\chi_1$ cannot be equal to $q^2$ by \cite[Proposition 4.4]{ASav} applied to the subfield generated by $f_A^{-1}(\chi_1)$, where $f_A: A^* \to A^*/\mathbb{K}^*$ is a canonical projection. Therefore the order of $\chi_1$ is $q$ and there exists an inclusion $\im \varphi \hookrightarrow G$, and thus $G$ is isomorphic either to $\mumu_n \rtimes_{\chi} \mumu_q$ or to
$$(\mumu_n \times \mumu_q) \rtimes_{\chi} \mumu_q \cong (\mumu_n \rtimes_{\chi} \mumu_q) \times \mumu_q.$$ 
The latter isomorphism is due to the fact that there is no non-trivial action of $\mumu_q$ on itself by automorphisms. Moreover, as discussed above, the semidirect products in both of these cases are Frobenius and balanced.

\end{proof}

Let us prove some lemmas before proving Theorems \ref{maincon} and \ref{mainconp}. Please pay attention that prime numbers $l$ and $q$ may be equal unless otherwise claimed precisely.

\begin{lm}
    \label{general}
    Let~$\mathbb{K}$ be a field such that~$\Char(\mathbb{K})= l$. Let~$L_1 \supset \mathbb{K}$ and~$L_2\supset \mathbb{K}$ be two field extensions such that~$\left[L_1: \mathbb{K}\right] = \left[L_2: \mathbb{K}\right] = q$ is a prime number. Let~$p_1 \neq p_2 \neq q$ be two different prime numbers. Let~$v_q(\ord_{p_1}(l)) = k_1$ and~$v_q(\ord_{p_2}(l)) = k_2$. Let~$k_1 \neq k_2$ and~$k_1,k_2 \geqslant 1$. Let~$n_1,n_2$ be two natural numbers. Let~$\theta_1 \in L_1$ and~$\theta_2 \in L_2$ such that~$\theta_1^{n_1p_1}, \theta^{n_2p_2}_2 \in \mathbb{K}$. Then~$\theta_1^{n_1} \in \mathbb{K}$ or~$\theta_2^{n_2} \in \mathbb{K}$.
\end{lm}

\begin{proof}
    Let~$\ord_{p_1}(l) = q^{k_1}m_1$ and~$\ord_{p_2}(l) = q^{k_2}m_2$. Suppose that $\theta_1^{n_1} \notin \mathbb{K}$ and~$\theta_2^{n_2} \notin \mathbb{K}$. Let~$k_1 >k_2$. By Lemma \ref{n} there are~${\alpha_1, \alpha_2 \in \mathbb{K}}$ such that $$\left(\frac{\theta_1^{n_1}}{\alpha_1}\right)^{p_1} - 1 = \left(\frac{\theta_2^{n_2}}{\alpha_2}\right)^{p_2} - 1 = 0.$$ Notice that decomposition field of polynomial~$x^{p_1} -1$ over~$\mathbb{F}_l$ is~$\mathbb{F}_{l^{q^{k_1}m_1}}$, so~${\frac{\theta_1^{n_1}}{\alpha_1} \in \mathbb{F}_{l^{q^{k_1}m_1}}}$. Moreover $$\mathbb{F}_l\left(\frac{\theta_1^{n_1}}{\alpha_1}\right) = \mathbb{F}_{l^{q^{k_1}m_1}}$$ because all finite field extensions of~$\mathbb{F}_l$ are normal. Then~$\mathbb{K}\left(\frac{\theta_1^{n_1}}{\alpha_1}\right)$ is a composite of~$\mathbb{K}$ and~$\mathbb{F}_{l^{q^{k_1}m_1}}$. By Theorem \ref{NI} $$\left[\mathbb{K}\left(\frac{\theta_1^{n_1}}{\alpha_1}\right) : \mathbb{K}\right] = \left[\mathbb{F}_{l^{q^{k_1}m_1}} : \mathbb{F}_{l^{q^{k_1}m_1}} \cap \mathbb{K}\right].$$ Since~$\mathbb{K} \subset \mathbb{K}\left(\frac{\theta_1^{n_1}}{\alpha_1}\right) \subset L_1$, then either~$\left[\mathbb{K}\left(\frac{\theta_1^{n_1}}{\alpha_1}\right) : \mathbb{K}\right] = 1$ or $$\left[\mathbb{K}\left(\frac{\theta_1^{n_1}}{\alpha_1}\right) : \mathbb{K}\right] = q.$$ In the first case one has that~$\theta_1^{n_1} \in \mathbb{K}$, a contradiction. So, only the second case is possible. In this case one has that $$\mathbb{F}_{l^{q^{k_1}m_1}} \cap \mathbb{K} = \mathbb{F}_{l^{q^{k_1-1}m_1}}.$$ Using exactly the same argument for~$\mathbb{K}\left(\frac{\theta_2^{n_2}}{\alpha_2}\right)$ one can get that $$\mathbb{F}_{l^{q^{k_2}m_2}} \cap \mathbb{K} = \mathbb{F}_{l^{q^{k_2-1}m_2}}.$$ Then~$\mathbb{K} \supset \mathbb{F}_{l^{q^{k_2-1}m_2}}$ and~$\mathbb{K} \supset \mathbb{F}_{l^{q^{k_1-1}m_1}}$. Hence $$\mathbb{K} \supset \mathbb{F}_{l^{q^{k_1-1}\lcm(m_1,m_2)}} \supset \mathbb{F}_{l^{q^{k_2}m_2}},$$ so one has a contradiction with~$\mathbb{F}_{l^{q^{k_2}m_2}} \cap \mathbb{K} = \mathbb{F}_{l^{q^{k_2-1}m_2}}$.
\end{proof}



\begin{cor}
    \label{cri}
    Let~$X$ be a non-trivial Severi--Brauer variety over a field~$\mathbb{K}$, let~${\Char(\mathbb{K}) = l}$ and~$\dim(X) = q-1$, where~$q$ is a prime number. Let~$f \in \Aut(X)$ such that~${f^n = \id}$, where~$n = \prod^m_{i=1}p^{r_i}_i$ and $n$ is not divisible by $q$. Let~$p_1$ and~$p_2$ be prime numbers from the decomposition of $n$ such that~$v_q(\ord_{p_1}(l)) \neq  v_q(\ord_{p_2}(l))$. Then~$f^{\frac{n}{p_1}} = \id$ or~$f^{\frac{n}{p_2}} = \id$.
\end{cor}

\begin{proof}
    Let $m = \frac{n}{p_1p_2}$ and $g = f^m$. Notice that $$g^{p_1p_2} = f^n = \id.$$
    Let $A$ be the central simple algebra which corresponds to $X$. Since ${\dim(X) = q - 1}$, where~$q$ is a prime number, $A$ is a division algebra. 
    By~Lemma~\ref{Aut} one has that~${\Aut(X)\cong A^*/ \mathbb{K}^*}$, so $g \in A^* / \mathbb{K}^*$. Denote by $u$ any lift of $g$ to $A^*$. Consider the field extension~$\mathbb{K} \subset \mathbb{K}(u)$. This extension is either trivial or has degree~$q$. In the first case one immediately has~$u \in \mathbb{K}^*$, so $$f^{\frac{n}{p_1p_2}} = \id,$$ and so $$f^{\frac{n}{p_1}} = f^{\frac{n}{p_2}} = \id.$$ In the second case by Lemma \ref{general} one has either $u^{p_1} \in \mathbb{K}^*$ and so $$g^{p_1} = f^{\frac{n}{p_2}} = \id,$$ or $u^{p_2} \in \mathbb{K}^*$ and so $$g^{p_2} = f^{\frac{n}{p_1}} = \id.$$ So~$f^{\frac{n}{p_1}} = \id$ or~$f^{\frac{n}{p_2}} = \id$.
\end{proof}

The following definition describes all sets $D = \{n \mid \mumu_n \subset \Aut(X)\}$, where $X$ is some non-trivial Severi--Brauer variety of dimension $q - 1$ over a field of positive characteristic~$l$.

\begin{defn}
\label{qconst}
Let $q$ be an odd prime number and $l \neq q$ another prime number. Let~${D \subset \mathbb{Z}_+}$ be a set of positive integers such that $\gcd(n,l) = 1$ for all $n \in D$. One says that a set $D$ is {\sf $q$-adically constant with respect to the order of $l$}, if there is a positive integer $k$ such that for all $n \in D$ one has that~$v_q(\ord_n(l)) = k$.
\end{defn}

\begin{proof}[Proof of Theorem \ref{maincon}]
  Suppose $\mumu_q \times (\mumu_n \rtimes_{\chi} \mumu_q) \in \frak{A}$, let us prove that $n \in \frak{I}_{k,q,l}$ for some positive integer $k$. Let~$n = \prod^m_{i=1}p^{r_i}_i$, where $p_i$ are prime numbers such that $p_i \equiv 1(\text{mod }q)$ for~${1 \leqslant i \leqslant m}$. Let~$f \in G$ be a generator of the cyclic group~$\mumu_n$. By Theorem \ref{ASavnew}, one has that all~$p_i \neq l$ for ${1 \leqslant i \leqslant m}$. Let~$p_i$ and~$p_j$ be prime numbers from the decomposition such that~${v_q(\ord_{p_i}(l)) \neq  v_q(\ord_{p_j}(l))}$. By Corollary \ref{cri} one has that either~$f^{\frac{n}{p_i}} = \id$ or~$f^{\frac{n}{p_j}} = \id$. Therefore~$f$ cannot be a generator of the group~$\mumu_n$. So $n \in \frak{I}_{k,q,l}$ and by Theorem \ref{ASavnew} one has that $k \geqslant 1$.
  
Let us choose another group $\mumu_q \times (\mumu_{n_1} \rtimes_{\chi_1} \mumu_q) \in \frak{A}$ such that~$n_1 = \prod^{m_1}_{i=1}t^{s_i}_i$, where $t_i$ are prime numbers such that $t_i \equiv 1(\text{mod }q)$ for~${1 \leqslant i \leqslant m_1}$. Suppose that there are~$1\leqslant i \leqslant m$ and~$1 \leqslant j \leqslant m_1$ such that~${v_q(\ord_{p_i}(l)) \neq v_q(\ord_{t_j}(l))}$. Let~$u_1$ be a generator of the group~$\mumu_{p_i}$ and~$u_2$ be a generator of the group~$\mumu_{t_j}$. Denote by $\theta_1$ any lift of~$u_1$ and by~$\theta_2$ any lift of $u_2$ to~$A^*$. Then by Lemma \ref{general} either~$\theta_1 \in \mathbb{K}$ or~$\theta_2 \in \mathbb{K}$ and this contradicts the fact that~$\theta_1$ and~$\theta_2$ are generators of cyclic groups. So for all~$1 \leqslant i \leqslant m$ and~$1 \leqslant j \leqslant m_1$ one has that~$v_q(\ord_{p_i}(l)) = v_q(\ord_{t_j}(l))$. Therefore~$n_1 \in \frak{I}_{k,q,l}$ for the same $k$ as $n$. So if $$D = \{ n \mid \mumu_n \subset \Aut(X) \text{ and }n\text{ is not divisible by }q\}$$ then $D$ is $q$-adically constant with respect to the order of $l$ by the Definition \ref{qconst}.  
  Now, by Theorem \ref{consistancy} and Theorem \ref{ASavnew}, one has that $\frak{A}$ is a Frobenius consistent balanced collection. 

 Finally let $G$ be a finite group acting on $X$, then by Theorem \ref{ASavnew} one has that~$G$ is a subgroup of $$\mumu_q \times (\mumu_n \rtimes_{\chi} \mumu_q)$$ for some~$n \in \frak{I}_{k,q,l}$ and $\chi$ is balanced and of Frobenius type.
\end{proof}

Now let us prove Theorem \ref{mainconp}. Firstly, let us provide some important lemmas. Let us follow the paper \cite{ASav} and again recall $\widetilde{N}_G$ as finite subgroup of $A^*$ of finite order lifts of elements $G$. Denote $N_G \subset G$ natural projection of the group $\widetilde{N}_G$.

\begin{thm}[{\cite[Corollary 2.2]{ASav}}]
\label{mainAnna}
Let $A$ be a central simple algebra over a field $k$ and let $G$ be a finite subgroup of~$A^*/k^*$. Then all the elements of $G$ that admit a lift of finite order in $A^*$ form a normal subgroup $N_G \lhd G$, such that $G/N_G$ is abelian.
\end{thm}

Let us research more about the structure of the groups $N_G$.

\begin{lm}
\label{critreriaq}
Let $q \geqslant 3$ be a prime number. Let $A$ be a central simple algebra over a field~$\mathbb{K}$ such that $\deg(A) = \Char(\mathbb{K}) = q$. Let $G$ be a finite subgroup of~$A^*/\mathbb{K}^*$. Let~${N_G \lhd G}$ be a subgroup as was described above. Let $g \in G$ and $p$ is the minimal positive integer such that $g^p \in \mathbb{K}^*$. If p is a prime number, then $g \in N_G$ if and only if~$q \neq p$.
\end{lm}  

\begin{proof}

Let $u \in A^*$ be any lift of $g$ to $A^*$. Notice that $u^p \in \mathbb{K}^*$. Notice that $\mathbb{K} \subset \mathbb{K}(u)$ is a field extension and $[\mathbb{K}(u) : \mathbb{K}] = q$, since $A$ is also an algebra over $\mathbb{K}(u)$. 

By Lemma \ref{n} if $q \neq p$ then there exists $\alpha \in \mathbb{K}$ such that $$\left(\frac{u}{\alpha}\right)^p - 1 = 0.$$ Therefore, $$\frac{u}{\alpha} \in \widetilde{N}_G$$ and so $g \in N_G$.

Let $q = p$. Suppose $g \in N_G$, then there exists some lift $v \in A^*$ of $g$ and positive integer~$m$ such that~$v^m = 1$, let $m$ be the minimal such integer. Then $g^m \in \mathbb{K}$, so $m$ is divisible by $q$. Let~$m = qm_0$, then $$(v^{m_0})^q - 1 = (v^{m_0} - 1)^q = 0,$$ and so~$v^{m_0} - 1 = 0$, then $m$ was not the minimal one, a contradiction. Therefore, $g \not \in N_G$.

\end{proof}

\begin{lm}
\label{triv}
Let $q \geqslant 3$ be a prime number. Let $A$ be a central simple algebra over a field~$\mathbb{K}$ such that $\deg(A) = \Char(\mathbb{K}) = q$. Let $G$ be a finite subgroup of~$A^*/\mathbb{K}^*$. Let~${N_G \lhd G}$ be a subgroup as was described above. If $N_G$ is trivial, then $G$ is isomorphic to $\mumu^N_q$ for some positive $N$.
\end{lm}

\begin{proof}
By Theorem \ref{mainAnna} one has that $G$ is abelian. Let $g \in G$ has order $m$. Let us prove that $m = q$, then the structure of $G$ will be obviously $\mumu^N_q$. If there is $p \neq q$ such that~$m$ is divisible by $p$, then we can apply Lemma \ref{critreriaq} to $g^{\frac{m}{p}}$ and get a contradiction with triviality of $N_G$. So the last possibility is that $m = q^l$. Let $u \in A^*$ be any lift of $g$ to $A^*$. Notice that $u^m \in \mathbb{K}^*$. Notice that $\mathbb{K} \subset \mathbb{K}(u)$ is a field extension and $[\mathbb{K}(u) : \mathbb{K}] = q$, since $A$ is also an algebra over $\mathbb{K}(u)$. Let $k_u = u^m \in \mathbb{K}^*$. Consider the polynomial $$f(x) = x^m - k_u \in \mathbb{K}[x].$$ If $k_u \not \in \mathbb{K}^q$, then by Theorem \ref{Leng9} one has that $f(x)$ is irreducible and since $[\mathbb{K}(u) : \mathbb{K}] = p$ one has that $m = q$. If $k_u = k_{u_1}^q$, then $$u^m - k_u = (u^{\frac{m}{q}} -  k_{u_1})^q = 0,$$ so $u^{\frac{m}{q}} -  k_{u_1} = 0$. Therefore $u^{\frac{m}{q}} \in \mathbb{K}^*$, which gives a contradiction with minimality of~$m$.
\end{proof}

\begin{lm}
\label{field}
Let $F/\mathbb{K}$ be a Galois extension, with $\Gal(F/\mathbb{K}) = \mumu_q$ and $\Char(\mathbb{K}) = q$. Let~$H$ be a finite subgroup of $F^*/\mathbb{K}^*$. Then there exists a positive integer $k$ and $n \in \frak{I}_{k,q,q}$ for some positive $k$ such that~$H$ is isomorphic to~$\mumu_n$. 
\end{lm}

\begin{proof}
Denote $\sigma$ as a generator of $\Gal(F/\mathbb{K})$. Let $g \in G$ be some element of the group, and let $\theta \in F$ be a lift of $g$. Then there is an integer $m$ such that $\theta^m \in \mathbb{K}$ and this $m$ is minimal. Consider $$\xi = \frac{\sigma(\theta)}{\theta}.$$ Notice that $$\xi^m = \frac{\sigma(\theta)^m}{\theta^m} = \frac{\sigma(\theta^m)}{\theta^m} = \frac{\theta^m}{\theta^m} = 1.$$ Therefore $\xi$ is root of unity of degree $m$. If there exists $m_0 < n$ such that $\xi^{m_0} = 1$, then $$1 = \xi^{m_0} = \frac{\sigma(\theta)^{m_0}}{\theta^{m_0}},$$ so $$\sigma^{t}(\theta^{m_0}) = \theta^{m_0}$$ for all integer $t$. It means that $\theta^{m_0} \in \mathbb{K}$ and it is a contradiction with minimality of $m$. Therefore $\xi$ is a primitive root of unity of degree $m$. Notice that by Lemma \ref{critreriaq} one has that $m$ is not divisible by $q$. Then by Lemma \ref{root} there is an element $\alpha \in \mathbb{K}$ such that $$\left(\frac{\theta}{\alpha}\right)^m = 1$$ and therefore $f$ can be lifted to an element $\frac{\theta}{\alpha} \in F$ which has an order exactly $m$. So the whole group $G$ can be lifted to an isomorphic $\widetilde{G} \subset F$ and since multiplicative group of a field is always cyclic, $\widetilde{G}$ will be also cyclic as a subgroup of a cyclic group and therefore~$G$ is cyclic.

So the last thing we need to prove is that the order of a generator $G$ lies in $\frak{I}_{k,q,q}$. Denote by $n$ the order of $G$. By Lemma \ref{general} one has that $v_q(\ord_p(q))$ is constant for all prime divisors of $n$.

Consider $\mathbb{K} \cap \mathbb{F}_q(\theta) = \mathbb{F}_{q^s}$, for some positive integer $s$ since $ \mathbb{F}_q(\theta)$ is a finite extension of $\mathbb{F}_q$. 

By Theorem \ref{NI} one has $$[F : \mathbb{K}] = [\mathbb{F}_q(\theta) : \mathbb{F}_{q^s}] = q.$$ Therefore $\mathbb{F}_q(\theta) = \mathbb{F}_{q^{qs}}$ for some integer $s$. One knows that $\mathbb{F}_q(\theta) = \mathbb{F}_{q^d}$, where $d = \ord_n{q}$. Therefore $\ord_n(q) = qs$. So $v_q(\ord_n(q))>0$, therefore there exists a positive integer $k$ such that $n \in \frak{I}_{k,q,q}$.


\end{proof}

Let us recall a result from group theory following the paper \cite{ASav}.

\begin{thm}[{\cite[Corollary 7.4]{ASav}}]
\label{cyc}
Let $N$ be a non-trivial finite group such that it has no non-cyclic Sylow $p$-subgroups for any prime $p$. Then $N$ contains a non-trivial cyclic characteristic subgroup.
\end{thm}

Now we are finally ready to prove Theorem \ref{mainconp}.

\begin{proof}[Proof of Theorem \ref{mainconp}]
Firstly, let us understand the structure of a finite group acting on~$X$, i.e., show the second part of the theorem. Let $G \subset \Aut(X)$ be a finite group and~$A$ be the central simple algebra corresponding to $X$. Consider the subgroup $N_G \lhd G$ as was described above. If $N_G$ is trivial then $G$ is isomorphic to $\mumu^N_q$ for some positive~$N$, by~Lemma \ref{triv}. If $N_G$ is nontrivial, then by Theorem \ref{cyc} there is a nontrivial cyclic subgroup of $N_G$. Let $g$ be a generator of this subgroup and $u$ its lift with finite order. Consider $F = \mathbb{K}(u) \supset \mathbb{K}$. Let $m$ be the minimal positive integer such that $u^m = 1$. Notice that $m$ is not divisible by $q$, since if $m = qm_0$, then $$u^m - 1 = (u^{m_0} - 1)^q = 0,$$ so $m$ is not the minimal one. Then $x^m - 1$ is separable, therefore $\mathbb{K}(u) \supset \mathbb{K}$ is normal and separable, and therefore $F/\mathbb{K}$ is a Galois extension. Since $A$ is an $F$-algebra, then~${[F : \mathbb{K}] = q}$, therefore $\Gal(F/\mathbb{K}) = \mumu_q$. 

Denote by $\phi_g$ an automorphism of $\mathbb{K}(u)$ which defined by $x \mapsto \widetilde{g}^{-1}x\widetilde{g}$, where $\widetilde{g}$ is any lift of $g$ to $A$. Since $\widetilde{g}^{-1} \mathfrak{k} \widetilde{g} = \mathfrak{k}$ for $\mathfrak{k} \in \mathbb{K}$ and ~$\widetilde{g}^{-1} u \widetilde{g} = \mathfrak{k}u^l$, for some integer~$l$ and some element~$\mathfrak{k} \in \mathbb{K}$, so it is actually an automorphism of $F$, which fixes $\mathbb{K}$. Therefore we have a map $\phi: G \to \Gal(F/\mathbb{K})$ which defineds~$g \to \phi_g$.  Consider $H = \Ker(\phi)$ and denote~$\widetilde{H}$ some lift of $H$ to $A$. If $h \in H$, then~$\mathfrak{f} \widetilde{h} = \widetilde{h} \mathfrak{f}$ for all~$\mathfrak{f} \in F$. Therefore~$\widetilde{H} \subset C_A(F)$, where~$C_A(F)$ is a centralizer of $F$ in $A$. Since $F$ is a maximal field, one has that $C_A(F) = F$ and so~$H \subset F$. By Lemma \ref{field}, one has that $H$ is isomorphic to $\mumu_n$, where $n \in \frak{I}_{k,q,q}$ for some positive integer $k$. Notice that for different groups $G$ the number $k$ should be the same, since it can be described by arithmetic properties of the base field~$\mathbb{K}$, i.e. $$k = \max\{l \mid \mathbb{K} \supset \mathbb{F}_{q^l}\} + 1.$$ In particularly one get that the set $D$ is $q$-adically constant with respect to the order of~$q$.
 
The only case that is left is then the image of $\phi$ and the kernel both are nontrivial. Let~$\alpha$ be an element that sends to the generator $\sigma$ of the Galois group. Then $\frak{f}\alpha = \sigma(\frak{f}) \alpha$ for all $\frak{f} \in F$, in particular for all $\frak{f} \in H$, so G is isomorphic to $H \rtimes \mumu_q$ and the semidirect product is balanced. Since we chose a lift~$u$ of a a generator $\mumu_n$ such that~${u^m = 1}$ then by Theorem \ref{NI} one has that $$\Gal(L/\mathbb{K}) \cong \Gal(\mathbb{F}_{q}(u)/\mathbb{K} \cap \mathbb{F}_{q}(u))$$ and since $\mathbb{F}_{q}(u)$ is a finite extension of $\mathbb{F}_l$ a generator of the group $\Gal(\mathbb{F}_{q}(u)/\mathbb{K} \cap \mathbb{F}_{q}(u))$ is $$\chi(y): x \to x^{q^{\left(d_0 \cdot d \cdot q^{(k-1)}\right)}},$$ for $y$ the generator of $\mumu_q$ and $x \in \mumu_n$, where $\ord_n(q) = q^kd_0$, such that $d_0$ is not divisible by $q$, and $1 \leqslant d \leqslant q -1$.

 By the last remark and Theorem \ref{consistancy}, one has that $\frak{A}$ is a Frobenius consistent balanced product.
\end{proof}

\section{Examples in positive characteristic}
\label{exchar}
In this section, we prove Theorems \ref{main} and \ref{mainp}. In particular, we will show that no more conditions on the finite subgroups can arise. To show it, firstly, we need to prove the following lemmas. The first one is a well-know Lifting-the-exponent lemma.

\begin{lm}[Lifting-the-exponent lemma]
    \label{lift}
    Let~$x$ and~$y$ be integers. Let~$n$ be a positive integer and let~$p$ be an odd prime number such that~$x$ and~$y$ are not divisible by~$p$, but~$x - y$ is divisible by~$p$. Then~$v_{p}(x^{n}-y^{n})= v_{p}(x-y)+v _{p}(n).$
\end{lm}

\begin{proof}

Let us start with a base case when~$n$ is not divisible by~$p$. Notice that $$x^n - y^n = (x-y)(x^{n-1} + x^{n-2}y + \ldots + y^{n-1})$$ and $$x^{n-1} + x^{n-2}y + \ldots + y^{n-1} \equiv n x^{n-1} \not \equiv 0 \text{ (mod }p).$$ So one has that $$v_p(x^{n}-y^{n})= v_{p}(x-y) = v_{p}(x-y) + v_{p}(n).$$

Now consider the case when~$n = p$. Let us choose~$k \in \mathbb{Z}$ such that~$y = x + kp$. Notice that for all~$1 \leqslant t <p$ one has $$y^tx^{p - 1 -t} = x^{p-1 -t}(x + kp)^{t} = x^{p-1 -t}(x^t + tkpx^{t-1} + \ldots) \equiv x^{p-1} + ptkx^{p-2} \text{ (mod }p^2).$$ It follows that $$x^{p-1} + x^{p-2}y + \ldots + y^{p-1} \equiv p(x^{p-1} + ptkx^{p-2}) \equiv px^{p-1} \text{ (mod }p^2),$$ and so $$v_p(x^{p}-y^{p})= v_{p}(x-y) + 1= v_{p}(x-y) + v_{p}(n).$$

Now, let us prove the general case. Let~$n = p^bn_1$ and~$n_1$ is not divisible by~$p$. Then by the base case one has $$v_p(x^{n}-y^{n})= v_{p}(x^{p^bn_1} - y^{p^bn_1}) = v_p({(x^{p^b})}^{n_1} - {(y^{p^b})}^{n_1}) = v_{p}(x^{p^b}-y^{p^b}).$$ Then by iterating the case~$n=p$ one has $$v_{p}(x^{p^b}-y^{p^b}) = v_{p}(x^{p^{b-1}}-y^{p^{b-1}}) + 1 = \ldots = v_p(x-y) + b = v_p(x-y) + v_p(n).$$

\end{proof}

\begin{lm}
    \label{lte}
    Let~$p$ be an odd prime number. Let $l$ be an integer such that~$l$ is not divisible by~$p$. Let~$\ord_p(l) = r$ and~$v_p(l^r -1) = s$. Then~$\ord_{p^{s + a}}(l) = p^ar$.
\end{lm}

\begin{proof}
   Notice that~$p-1$ is divisible by~$r$, so~$r$ is not divisible by~$p$.

    Using this Lemma \ref{lift} in the case~$n = p^a$,~$x = l^r$ and~$y = 1$ one has
    $$v_p(l^{p^ar}-1) = v_p(l^{r} - 1) + v_p(p^a) = s+a.$$
\end{proof}

\begin{lm}
    \label{Gal}
    Let~$l$ be a prime number. Let~$n$ be some integer. Let~$p$ be a prime number such that $p \neq l$ and let~$\ord_p(l) = r$. Let~$L$ be a field extension of~$\mathbb{F}_{l^n}$ by all roots of unity of degree~$p^{t}$, i.e.~$L = \mathbb{F}_{l^n} (\xi_p, \xi_{p^2},\ldots)$, where~$\xi_{p^t}$ is a root of unity of degree~$p^{t}$. Then $$\Gal(L/\mathbb{F}_{l^n}) = \mumu_{\frac{r}{\gcd(r,n)}} \times \mathbb{Z}_{p},$$ where $\mathbb{Z}_p$ denotes the group of $p$-adic integers.
\end{lm}

\begin{proof}
    Let~$v_p(l^r -1) = s$. Notice that decomposition field of polynomial~$x^m -1$ over~$\mathbb{F}_{l^{n}}$ is~$\mathbb{F}_{l^{\lcm(n,k)}}$, where~$k = \ord_m (l)$. Then by Lemma \ref{lte} the decomposition field of the polynomial~$x^{p^{{s+ a}}} - 1$ over~$\mathbb{F}_{l^n}$ is~$\mathbb{F}_{l^{\lcm(n,rp^a)}}$. By construction the Galois group of the extension $\mathbb{F}_{l^n}\subset L$ is a profinite group, which is realized as an inverse limit of Galois groups of extensions $$\mathbb{F}_{l^n}(\xi_{p^{s+a}})/\mathbb{F}_{l^n} = \mathbb{F}_{l^{\lcm(n,rp^a)}}/\mathbb{F}_{l^n}.$$ So $$\Gal(\mathbb{F}_{l^n}(\xi_{p^{s+a}})/\mathbb{F}_{l^n}) = \Gal(\mathbb{F}_{l^{\lcm(n,rp^a)}}/\mathbb{F}_{l^n}) = \mumu_{\frac{rp^a}{\gcd(n,rp^a)}}.$$ Therefore $$\Gal(L/\mathbb{F}_{l^n}) = \varprojlim \mumu_{\frac{rp^a}{\gcd(n,rp^a)}} = \mumu_{\frac{r}{\gcd(r,n)}} \times \mathbb{Z}_p.$$
\end{proof}


\begin{proof}[Proof of Theorem \ref{main}]
    Let~$d_0 = \ord_q(l)$. So~$q - 1$ is divisible by~$d_0$, and therefore~$d_0$ is not divisible by~$q$. Let $D = \{n \in \mathbb{Z} \mid \mumu_q \times (\mumu_n \rtimes_{\chi}\mumu_q) \in \frak{A} \}$ for some Frobenius balanced map~$\chi$. Notice that by the definition, there is a positive integer $k$ such that for all $n \in D$ one has that $n$ is divisible only by primes $p$ congruent to~$1$ modulo~$q$, such that~${v_q(\ord_{p}(l)) = k}$. 
    
    Consider the field extension~$L$ of~$\mathbb{F}_{l^{d_0}}$ by all roots of unity of degree~$p^r$ where~$p^r \in D$ and $p$  is a prime number and~$r$ is some arbitrary integer. More precisely, let us enumerate all pairs~$(p,r)$, where~$p$ is a prime number such that~$p \equiv 1 \text{ (mod }q$) and~$r$ is a positive integer, such that $p^r \in D$, by positive integers (if the set $D$ is finite, then enumeration is by first several integers). Let~$(p,r)$ corresponds to~$t$. Then let us denote~$\zeta_t = \xi_{p^r}$, where~$\xi_{p^r}$ is a root of unity of degree~$p^r$. Then~$L = \mathbb{F}_{l^{d_0}}( \zeta_1, \zeta_2, \zeta_3,\ldots)$. If~$D$ was a finite set, then $L$ is a finite extension of $\mathbb{F}_{l^{d_0}}$, otherwise it is an infinite extension. Let us denote $$B_j = \{p \mid \exists r \text{ and } n \in \mathbb{N} \text{ such that } n \leqslant j \text{ and } n \text{ correspond to } (p,r) \},$$  in other words it is a set of all prime numbers which appears in the first $j$ pairs. And $$C_{j,p} = \{r\mid \exists n \in \mathbb{N} \text{ such that } n \leqslant j \text { and } n \text{ correspond to } (p,r) \},$$ in other words it is a set of powers in which $p$ appears in the first $j$ pairs. Let $r_j(p)$ be the maximum of the set~$C_{j,p}$. Let us denote $$L_j = \mathbb{F}_{l^{d_0}}(\zeta_1, \zeta_2, \ldots ,\zeta_j).$$ Notice that $$L_j = \mathbb{F}_{l^{d_0}}(\xi_{N(j)}),$$ where~$\xi_{N(j)}$ is root of unity of degree~$N(j) = \prod_{p \in B_j}p^{r_j(p)}$.
        
    Notice that the decomposition field of polynomial~$x^p -1$ over~$\mathbb{F}_{l^{d_0}}$ is~$\mathbb{F}_{l^{\lcm(d,d_0)}}$, where~${d = \ord_p (l)}$ and the Galois group of this extension is~$\mumu_{\frac{d}{\gcd(d,d_0)}}$. Since~$d_0$ is not divisible by~$q$, so $$v_q\left(\frac{d}{\gcd(d,d_0)}\right) = v_q (d).$$ 
    
    By Lemma \ref{Gal} applied for all $p \in D$ which are prime numbers independently, the Galois group of the extension $L/\mathbb{F}_{l^{d_0}}$ is $$\prod \mumu_{q^k} \times \prod \mumu_N \times \prod\limits_{\substack{p\in D\\ p \text{ prime}}}\mathbb{Z}_p,$$ where in~$\prod \mumu_N$ all ``tails'' from different cyclic groups were put. More precisely, it is a collection of cyclic groups that appears from Lemma \ref{Gal}, which was applied several times, after removing the cyclic group $\mumu_{q^k}$. 
    
Consider $$\sigma = \epsilon^{q^{k-1}}_{q^k} \times  \epsilon^{q^{k-1}}_{q^k} \times \ldots  \times 1 \times 1 \times \ldots \times 0 \times 0 \times \ldots \in \prod \mumu_{q^k} \times \prod \mumu_N \times \prod\limits_{\substack{p\in D\\ p \text{ prime}}}\mathbb{Z}_p,$$ where~$\epsilon_n$ is a generator of a group~$\mumu_n$, such that for all $n \in D$ if $\xi_n \in L_j$ then one has that $$\sigma(\xi_n) = \sigma|_{L_j}(\xi_n) = \chi(\xi),$$ where $\mumu_n \rtimes_{\chi}\mumu_q \in \frak{A}$. Since $\frak{A}$ is a Frobenius balanced collection, one can choose such an element $\sigma$.

Consider a group $\Delta$ generated by $\sigma$. Notice that~$\Delta \simeq \mumu_q$ is a subgroup of order~$q$ of $$\prod \mumu_{q^k} \times \prod \mumu_N \times \prod\limits_{\substack{p\in D\\ p \text{ prime}}}\mathbb{Z}_p.$$ Since~$\Delta$ is a finite subgroup, it is a closed subgroup of~$\Gal(L/\mathbb{F}_{l^{d_0}})$. 
    
    It means that there is~$\Bbbk = L^{\Delta}$ and~$\left[L:\Bbbk\right] = q$. 
    
 One can construct by the field extension $L(t)/\Bbbk(t)$ a non-trivial Severi--Brauer variety~$X$. Let~${A=(L(t)/\Bbbk(t),\sigma,t)}$ be a corresponding to the variety central simple algebra. Then by Lemma \ref{Aut} one has~${\Aut(X) = A^*/\Bbbk(t)^*}$. We claim that $X$ is the variety we want, i.e., such a variety that any group $G \in \frak{A}$ is acting on $X$.
    
Fix a positive integer $m$ and some non-negative integers $r_1,r_2,\ldots,r_m$. Consider~$m$ different prime numbers~$p_1,\ldots,p_m$, such that $v_q(\ord_{p_i}(l)) = k$ for $1 \leqslant i \leqslant m$. Let~${n = \prod^m_{i = 1}p^{r_i}_i}$. Suppose $n \in D$.

Consider $${\theta = \xi_{p^{r_1}_{{1}}} \cdot \ldots\cdot \xi_{p^{r_m}_{{m}}}}.$$ Since $n \in D$ then $\theta \in L$. Therefore $\theta \in L_j$ for some $j$. Notice that $\sigma(\theta) = \sigma|_{L_j}(\theta)$ by the construction of the Galois group. Since $L_j$ is a finite extension of $\mathbb{F}_{l^{d_0}}$, therefore $\sigma|_{L_j}$ is a power of a Frobenius map, i.e. $$\sigma|_{L_j} (\theta) = \theta^{l^{\left(d_0 \cdot d \cdot q^{(k-1)}\right)}},$$ where $0 < d < q$ is some integer. Moreover by the construction $\theta^{l^{\left(d_0 \cdot d \cdot q^{(k-1)}\right)}} = \chi(\theta),$ where~$\mumu_n \rtimes_{\chi}\mumu_q \in \frak{A}$.

Let~$1 \leqslant x < n$. If~$\theta^x \in \Bbbk$ then~$\theta^x$ is fixed by~$\sigma$, and so by $\sigma|_{L_j}$ this means that $$(\theta^x)^{l^{\left(d_0 \cdot d \cdot q^{(k-1)}\right)}} = \theta^x.$$ So $$x ( l^{\left(d_0 \cdot d \cdot q^{(k-1)}\right)} -1) \text{ is divisible by } n,$$ and so there exists~$p_{{h}}$ for some $1 \leqslant h \leqslant m$ such that $$l^{\left(d_0 \cdot d \cdot q^{(k-1)}\right)} -1 \text{ is divisible by }p_{{h}}.$$ This gives a contradiction with~$k = v_q( \ord_{p_{{h}} }(l))$, because both $d_0$ and $d$ are not divisible by $q$. So the order~$\theta$ in~$A^*/\Bbbk(t)^*$ is exactly~$n$. 

Let $G_n \subset A^*/\Bbbk(t)^*$ be the subgroup generated by images~$\theta$ and~$\alpha$, where $\alpha$ is the standard generator of the cyclic algebra $A$ (so $\alpha^q = t$). Let us prove that $G_n$ is isomorphic to~${\mumu_n \rtimes_{\chi} \mumu_q}$, where the semidirect product is Frobenius and balanced and~${\mumu_q \times (\mumu_n \rtimes_{\chi} \mumu_q) \in \frak{A}}$. The order of~$\theta$ as we proved above is~$n$ and the order of~$\alpha$ is~$q$. By definition one has $$\theta \alpha = \alpha \sigma(\theta) = \alpha \sigma|_{L_j}(\theta) = \alpha \theta^{l^{\left(d_0 \cdot d \cdot q^{(k-1)}\right)}} = \alpha \chi(\theta).$$

Finally, let~$\gamma_1$ be a generator of the cyclic group $(\mathbb{F}_{l^{q^kd_0}})^*$. Notice that $$(\mathbb{F}_{l^{q^kd_0}})^* \subset L_1 = \mathbb{F}_{l^{d_0}}(\zeta_1),$$ since $\zeta_1 = \xi_{p^r}$ for some positive $r$ and some $p$ prime number such that $v_q(\ord_p(l)) = k$ and~$p \in D$. Suppose that $0 < d < q$ is an integer is such that $$\sigma|_{L_1}: x \to x^{l^{\left(d_0 \cdot d \cdot q^{(k-1)}\right)}}.$$ 

Consider $\gamma = \gamma_1^{z}$, where $$z = \frac{l^{\left(d_0 \cdot d \cdot q^k\right)} - 1}{q(l^{\left(d_0 \cdot d \cdot q^{(k-1)}\right)} - 1)}.$$ By Lemma \ref{lift}, this is an integer.

Since $$(\gamma^q)^{l^{\left(d_0 \cdot d \cdot q^{(k-1)}\right)}} = (\gamma_1^{qz})^{l^{\left(d_0 \cdot d \cdot q^{(k-1)}\right)}} = \gamma_1^{qz + l^{\left(d_0 \cdot d \cdot q^k \right)} - 1} = \gamma^q,$$ it follows that $\gamma^q \in \Bbbk$. Notice that for $1 \leqslant y < q$ one has  $$(\gamma^y)^{l^{\left(d_0 \cdot d \cdot q^{(k-1)}\right)} -1} = (\gamma_1^{yz})^{l^{\left(d_0 \cdot d \cdot q^{(k-1)}\right)} - 1} =\gamma_1^{y\frac{l^{\left(d_0 \cdot d \cdot q^k\right)} - 1}{q}}.$$ Since $\gamma_1$ is a generator one has $$\gamma_1^{y\frac{l^{\left(d_0 \cdot d \cdot q^k\right)} - 1}{q}} \neq 1,$$ since by Lemma \ref{lift} one has 
\begin{multline*}
v_q\left(y\frac{l^{\left(d_0 \cdot d \cdot q^k\right)} - 1}{q}\right) = v_q(y(l^{\left(d_0 \cdot d \cdot q^k\right)} - 1)) - 1 =  \\ = v_q(l^{\left(d_0 \cdot d \cdot q^k\right)} - 1) - 1 = v_q(l^{\left(d_0 \cdot q^k\right)} - 1) - 1 < v_q(l^{\left(d_0 \cdot q^k\right)} - 1).$$ 
\end{multline*}
Therefore $$(\gamma^y)^{l^{q^{k-1}d_0}} \neq \gamma^y.$$ This means that the image $\gamma$ has order $q$ in~$A^*/\Bbbk(t)^*$. 

The elements~$\theta$ and~$\gamma$ commute because they both lie in the field $L(t)$. By construction one has $$\gamma \alpha = \alpha \sigma(\gamma) = \alpha \sigma|_{L_1}(\gamma) = \alpha \gamma^{l^{\left(d_0 \cdot d \cdot q^{(k-1)}\right)}}.$$ Notice that $$\gamma^{l^{\left(d_0 \cdot d \cdot q^{k - 1}\right)} -1} = \gamma_1^{\frac{l^{\left(d_0 \cdot d \cdot q^k\right)} - 1}{q}}$$ and $$\left(\gamma_1^{\frac{l^{\left(d_0 \cdot d \cdot q^k\right)} - 1}{q}}\right)^{l^{q^{k-1}d_0}-1} = 1,$$ since $l^{q^{k-1}d_0} - 1$ is dividing by $q$ (remain that $d_0 = \ord_q(l)$). 

Therefore $$\gamma^{l^{\left(d_0 \cdot d \cdot q^{k - 1}\right)}} = \epsilon\gamma,$$ where $$\epsilon = \gamma_1^{\frac{l^{\left(d_0 \cdot d \cdot q^k\right)} - 1}{q}} \in \Bbbk.$$ Therefore elements~$\alpha$ and~$\gamma$ commute in~$A^*/\Bbbk(t)^*$. This means that the group~$G(n)$ generated by images~$\alpha$,~$\theta$ and~$\gamma$ in $A^*/\Bbbk(t)^*$ is isomorphic to~${\mumu_q \times G_n \cong \mumu_q \times (\mumu_n \rtimes_{\chi} \mumu_q) \in \frak{A}}$. 

The integer $m$ and the integers~$r_1,\ldots,r_m$ are arbitrary. Hence, for any $$n = \prod^{m}_{i=1}p^{r_i}_i \in D,$$ there is a subgroup $G(n)$ of $\Aut(X)$ such that~${G(n) \simeq \mumu_q \times (\mumu_n \rtimes_{\chi} \mumu_q).}$ This finishes the proof of the theorem.

\end{proof}

And finally, let us prove Theorem \ref{mainp}

\begin{proof}[Proof of Theorem \ref{mainp}]
    Let $D = \{n \in \mathbb{Z} \mid \mumu_n \rtimes_{\chi}\mumu_q \in \frak{A} \}$ for some Frobenius balanced map. Notice that by the definition there is a positive integer $k$ such that for all $n \in D$ one has that $n$ is divisible only by primes $p$ congruent to~$1$ modulo~$q$, such that ${v_q(\ord_{p}(q)) = k}$. 
    
    Consider the field extension~$L$ of~$\mathbb{F}_{q}$ by all roots of unity of degree~$p^r$ where~$p^r \in D$ and~$p$  is a prime number and~$r$ is some arbitrary integer. More precisely, let us enumerate all pairs~$(p,r)$, where~$p$ is a prime number such that~$p \equiv 1 \text{ (mod }q$) and~$r$ is a positive integer, such that $p^r \in D$, by positive integers (if the set $D$ is finite, then enumeration is by first several integers). Let~$(p,r)$ corresponds to~$t$. Then let us denote~$\zeta_t = \xi_{p^r}$, where~$\xi_{p^r}$ is a root of unity of degree~$p^r$. Then~$L = \mathbb{F}_{q}( \zeta_1, \zeta_2, \zeta_3,\ldots)$. If~$D$ was a finite set, then $L$ is a finite extension of $\mathbb{F}_{q}$, otherwise it is a infinite extension. Let us denote $$B_j = \{p \mid \exists r \text{ and } n \in \mathbb{N} \text{ such that } n \leqslant j \text{ and } n \text{ correspond to } (p,r) \},$$  in other words it is a set of all prime numbers which appears in the first $j$ pairs. And $$C_{j,p} = \{r\mid \exists n \in \mathbb{N} \text{ such that } n \leqslant j \text { and } n \text{ correspond to } (p,r) \},$$ in other words it is a set of powers in which $p$ appears in the first $j$ pairs. Let $r_j(p)$ be the maximum of the set~$C_{j,p}$. Let us denote $$L_j = \mathbb{F}_{q}(\zeta_1, \zeta_2, \ldots ,\zeta_j).$$ Notice that $$L_j = \mathbb{F}_{q}(\xi_{N(j)}),$$ where~$\xi_{N(j)}$ is root of unity of degree~$N(j) = \prod_{p \in B_j}p^{r_j(p)}$.
        
    Notice that the decomposition field of polynomial~$x^p -1$ over~$\mathbb{F}_{q}$ is~$\mathbb{F}_{q^d}$, where~${d = \ord_p (q)}$ and the Galois group of this extension is~$\mumu_{d}$. 
        
    By Lemma \ref{Gal} the Galois group of this extension is $$\prod \mumu_{q^k} \times \prod \mumu_N \times \prod\limits_{\substack{p\in D\\ p \text{ prime}}}\mathbb{Z}_p,$$ where in~$\prod \mumu_N$ all ``tails'' from different cyclic groups were put. More precisely, it is a collection of cyclic groups that appears from Lemma \ref{Gal}, which was applied several times, after removing the cyclic group $\mumu_{q^k}$. 
    
Consider $$\sigma = \epsilon^{q^{k-1}}_{q^k} \times  \epsilon^{q^{k-1}}_{q^k} \times \ldots  \times 1 \times 1 \times \ldots \times 0 \times 0 \times \ldots \in \prod \mumu_{q^k} \times \prod \mumu_N \times \prod\limits_{\substack{p\in D\\ p \text{ prime}}}\mathbb{Z}_p,$$ where~$\epsilon_n$ is a generator of a group~$\mumu_n$, such that for all $n \in D$ if $\xi_n \in L_j$ one has that $$\sigma(\xi_n) = \sigma|_{L_j}(\xi_n) = \chi(\xi),$$ where $\mumu_n \rtimes_{\chi}\mumu_q \in \frak{A}$. Since $\frak{A}$ is a Frobenius balanced collection one can choose such element $\sigma$.

Consider a group $\Delta$ generated by $\sigma$. Notice that~$\Delta \simeq \mumu_q$ is a subgroup of order~$q$ of $$\prod \mumu_{q^k} \times \prod \mumu_N \times \prod\limits_{\substack{p\in D\\ p \text{ prime}}}\mathbb{Z}_p.$$ Since~$\Delta$ is a finite subgroup, it is a closed subgroup of~$\Gal(L/\mathbb{F}_{q})$. 
    
    It means that there is~$\Bbbk = L^{\Delta}$ and~$\left[L:\Bbbk\right] = q$. 
    
    One can construct by the field extension $L(t)/\Bbbk(t)$ a non-trivial Severi--Brauer variety~$X$. Let~${A=(L(t)/\Bbbk(t),\sigma,t)}$ be a corresponding to the variety central simple algebra. Then by Lemma \ref{Aut} one has~${\Aut(X) = A^*/\Bbbk(t)^*}$. We claim that $X$ is the variety we want, i.e., such a variety that any group $G \in \frak{A}$ is acting on $X$.
        
Fix a positive integer $m$ and some non-negative integers $r_1,r_2,\ldots,r_m$. Consider~$m$ different prime numbers~$p_1,\ldots,p_m$, such that $v_q(\ord_{p_i}(q)) = k$ for $1 \leqslant i \leqslant m$. Let~${n = \prod^m_{i = 1}p^{r_i}_i}$. Suppose $n \in D$.

Consider $${\theta = \xi_{p^{r_1}_{{1}}} \cdot \ldots\cdot \xi_{p^{r_m}_{{m}}}}.$$ Since $n \in D$ then $\theta \in L$. Therefore $\theta \in L_j$ for some $j$. Notice that $\sigma(\theta) = \sigma|_{L_j}(\theta)$ by the construction of the Galois group. Since $L_j$ is a finite extension of $\mathbb{F}_{q}$, therefore $\sigma|_{L_j}$ is a power of a Frobenius map, i.e. $$\sigma|_{L_j} (\theta) = \theta^{q^{\left(d \cdot q^{(k-1)}\right)}},$$ where $0 < d < q$ is some integer. Moreover by the construction $\theta^{q^{\left(d \cdot q^{(k-1)}\right)}} = \chi(\theta),$ where~$\mumu_n \rtimes_{\chi}\mumu_q \in \frak{A}$.

Let~$1 \leqslant x < n$. If~$\theta^x \in \Bbbk$ then~$\theta^x$ is fixed by~$\sigma$, and so by $\sigma|_{L_j}$ this means that $$(\theta^x)^{q^{\left(d \cdot q^{(k-1)}\right)}} = \theta^x.$$ So $$x ( q^{\left(d \cdot q^{(k-1)}\right)} -1) \text{ is divisible by } n,$$ and so there exists~$p_{{h}}$ for some $1 \leqslant h \leqslant m$ such that $$q^{\left(d \cdot q^{(k-1)}\right)} -1 \text{ is divisible by }p_{{h}}.$$ This gives a contradiction with~$k = v_q( \ord_{p_{{h}} }(q))$, because $d$ is not divisible by $q$. So the order~$\theta$ in~$A^*/\Bbbk(t)^*$ is exactly~$n$. 

Let $G_n \subset A^*/\Bbbk(t)^*$ be the subgroup generated by images~$\theta$ and~$\alpha$, where $\alpha$ is the standard generator of the cyclic algebra $A$ (so $\alpha^q = t$). Let us prove that $G_n$ is isomorphic to~${\mumu_n \rtimes_{\chi} \mumu_q \in \frak{A}}$, where $\chi$ is a balanced map of Frobenius type. The order of~$\theta$ as we proved above is~$n$ and the order of~$\alpha$ is~$q$. By definition one has $$\theta \alpha = \alpha \sigma(\theta) = \alpha \sigma|_{L_j}(\theta) = \alpha \theta^{q^{\left(d \cdot q^{(k-1)}\right)}} = \alpha \chi(\theta).$$

The integer $m$ and the integers~$r_1,\ldots,r_m$ are arbitrary. Hence, for any $$n = \prod^{m}_{i=1}p^{r_i}_i \in D,$$ there is a subgroup $G(n)$ of $\Aut(X)$ such that~${G(n) \cong \mumu_n \rtimes_{\chi} \mumu_q.}$ 

Finally, notice that all elements $\alpha + \frak{k}$, where $\frak{k} \in \Bbbk(t)$ has order $q$, since $$(\alpha + \frak{k})^q = \alpha^q + \frak{k}^q \in \Bbbk(t)$$ and by the definition of cyclic algebra $(\alpha + \frak{k})^t \notin \Bbbk(t)$. Since $\Bbbk(t)$ is an infinite field and all elements of type $\alpha + \frak{k}$ commute, for any positive integer $N$ there is a group $\mumu_q^N$ which acts on $X$. One of the examples is the group generated by $\alpha +t, \alpha + t^2, \ldots , \alpha + t^N$ which is isomorphic exactly to $\mumu_q^N$. This finishes the proof of the theorem.

\end{proof}

\section{Comments}
\label{com}
In this section, several observations are made.

\begin{st}
    \label{trfin}
    Let $l$ be a prime number. Let $\mathbb{K} \subset \overline{\mathbb{F}_l}$ be a field. Then there is no non-trivial Severi--Brauer variety over $\mathbb{K}$.
\end{st}

\begin{proof}
    Let $X$ be a Severi--Brauer variety over $\mathbb{K}$. Notice that $X$ is actually defined over some finite extension $\mathbb{L}$ of $\mathbb{F}_l$. By \cite[Proposition 6.2.3 and Theorem 6.2.6]{G-S}, one has that any Severi--Brauer variety over $\mathbb{L}$ is trivial. Thus $X$ is trivial. 
\end{proof}

One may wonder if there is only a finite set of different cyclic groups which act on the fixed Severi--Brauer variety. By the construction from Theorem \ref{example}, there are Severi--Brauer varieties such that there are infinitely many cyclic groups which act on the varieties. But it is still a question in the case of positive characteristic. 

\begin{defn}
   \label{set}
    Let~$A_{q,k}$ be a set of all prime numbers~$p$ such that~$v_q(\ord_p(l)) = k$, where~$q$ is a prime number and~$k \geqslant 1$.
\end{defn}

The following observation shows that the set $A_{q,k}$ is infinite.
\begin{st}
\label{infin}
    Let $q$ and $l$ be two prime numbers. Let $k$ be any positive integer. Then there is an infinite set $B_{q,k}$ of prime numbers~$p$ such that \begin{equation}
        v_q(p-1) = k
        \label{first}
    \end{equation} and \begin{equation}
        v_q(\ord_p(l)) \geqslant k.
        \label{second}
    \end{equation}
\end{st}

\begin{proof}
If $l \in B_{q,k}$, one may remove $l$ from $B_{q,k}$ without affecting finiteness, so assume~${l \notin B_{q,k}}$. Notice that if $p \neq l$ then $$v_q(\ord_p(l)) \leqslant  v_q(p-1) = k,$$ so for $p \in B_{q,k}$ one has that $v_q(\ord_p(l)) = k$, therefore $p \in A_{q,k}.$ Therefore $B_{q,k} \subset A_{q,k}$.

 Consider the field extension~$\mathbb{Q} \subset \mathbb{Q}(\epsilon_{q^{k+1}}, l^{\frac{1}{q}})$, where $\epsilon_{q^{k+1}}$ is a primitive root of unity of degree $q^{k+1}$. This is a Galois extension with Galois group~$(\mumu_{q^{k+1}})^* \ltimes \mumu_{q}$. Consider a Frobenius element $\mathfrak{p}$ corresponding to the ideal generated by a prime number~$p$ (for more details see~\cite[Chapter 8]{Miln}). Let us say that in the Galois group $\mathfrak{p}$ corresponds to~$(p, \alpha_p)$. The first condition \eqref{first} on~$p$ rewrites as~$p = 1 + q^km$ and~$p \neq 1 + q^{k+1}m$, which are just conditions on the first term of $\mathfrak{p}$ in the Galois group. The second condition~\eqref{second} can be rewritten as $$l^{\frac{p-1}{q}} \not \equiv 1 (\text{mod }p).$$ Let the Frobenius element acts on $l^{\frac{1}{q}}$. On the one hand it is just~$l^{\frac{p}{q}}$ modulo ideal~$(p)$, on the other hand it is $$(\epsilon_{q^{k+1}}^{q^k})^{\alpha_p}l^{\frac{1}{q}}.$$ So the second condition just means that $$(\epsilon_{q^{k+1}}^{q^k})^{\alpha_p} \not\equiv 1 (\text{mod } p).$$ And this is just~$\alpha_p \neq 0$ modulo~$q$, which gives us the second condition on the second term of $\mathfrak{p}$. By Chebotarev's density theorem (see \cite[Theorem 8.31]{Miln}), there is a set of some positive measure of prime numbers which satisfies both conditions, so this set is infinite.

\end{proof}

\appendix
\section{An application to $G$-birationally rigidty}
\label{Ap}
In this section, we apply the results of classification of finite groups acting on non-trivial Severi--Brauer varieties of dimension $q -1$, where $q$ is a prime number. The main result here is that any non-trivial  Severi--Brauer $G$-variety over a field $\mathbb{K}$ with dimension~${q -1}$ is not $G$-birationally rigid if $q \neq \Char(\mathbb{K})$ and $q$ is a prime number. 

Let us define a $G$-birationally rigid variety. Recall the following common definitions.

\begin{defn}[{\cite[Definition 1.1.5]{Ch-S}}]
A variety $X$ endowed with an action of a finite group $G$ and a $G$-equivariant surjective morphism $\phi: X \to S$ with connected fibers to some variety $S$ with an action of $G$ is a {\sf $G$-Mori fiber space} if 

\begin{enumerate} 
\item[(A)] it has terminal singularities,
\item[(B)] all $G$-invariant Weil divisors on $X$ are Cartier divisors,
\item[(C)]  the dimension of $S$ is strictly smaller than the dimension of $X$,
\item[(D)]  the anticanonical class $-K_X$ is $\phi$-ample,
\item[(E)] $\rk\Pic(X)^G = \rk \Pic (S)^G + 1$.
\end{enumerate}  

If $S$ is a point, then $X$ is called a {\sf $G$-Fano variety}.
\end{defn}


\begin{defn}[{\cite[Definition 3.1.1]{Ch-S}}]
The $G$-Fano variety $X$ is called {\sf $G$-birationally rigid} if the following two conditions are satisfied:
\begin{enumerate}
\item[(A)] there is no $G$-birational map $X \dashrightarrow Y$ such that $Y$ is a $G$-Mori fiber space over a positive dimensional variety;
\item[(B)] if there is a $G$-birational map $\xi : X \dashrightarrow  Y$ such that $Y$ is a $G$-Fano variety, then

   \begin{enumerate}
      \item[(B1)] the variety $Y$ is isomorphic to $X$;
      \item[(B2)] there is a $G$-birational map $\rho \in \Bir^G(X)$ such that the composition~${\xi \circ \rho: X \dashrightarrow  Y}$ is a biregular $G$-morphism.
   \end{enumerate}
   
\end{enumerate}

\end{defn}

For more details, see, for example, \cite[Chapter 3]{Ch-S}. Note that for any finite group $G$ acting on the Severi--Brauer variety $X$, one has that~$X$ is a $G$-Fano variety, since $X$ is smooth and $\rk\Pic X = 1$.

Recall a useful fact about Severi--Brauer varieties.
\begin{lm}[{\cite[Corollary 5.3.5]{G-S}}]
\label{mini}
Let $X$ be a Severi--Brauer variety and $A$ be a central simple algebra corresponding to $X$. Then $A$ is a division algebra if and only if $X$ is minimal, i.e., there is no twisted-linear subvariety in $X$.
\end{lm}

By Wedderburn's theorem (see e.g.~{\cite[~Theorem~2.1.3]{G-S}}) for any central simple algebra~$A$ there is a division algebra $D$ and an integer $n \geqslant 1$, such that $A \cong \Mat_n(D)$. Therefore if $\deg(A) = q$, where $q$ is a prime number, then $A$ is either a division algebra or a matrix algebra over the base field. This means that non-trivial Severi--Brauer varieties of dimension $q - 1$ correspond to division algebras and so by Lemma \ref{mini} these Severi--Brauer varieties are minimal. 

Before proving any statements, let us give the following definition:

\begin{defn}
Let $X$ be a Severi--Brauer variety over a field $\mathbb{K}$. Let $P$ be a point on~$X$. We call $P$ {\sf a general point} if there is no proper twisted-linear subvariety containing $P$.
\end{defn}

Recall that a point $P$ is separable if its residue field is a separable extension of the base field. 

\begin{st}
\label{lembir}
Let $q \geqslant 3$ be a prime number. Let~$X$ be a non-trivial Severi--Brauer variety of dimension~$q - 1$ over a field~$\mathbb{K}$, where $\Char(\mathbb{K}) = l$ and $l \neq q$ (it is possible that~$l = 0$). Let $G$ be a finite group acting on $X$. Then there is a separable general point of degree $q$ on $X$ which is fixed by all elements of the group $G$.
\end{st}

\begin{proof}
By Theorem \ref{Sav} there exists an integer $n$ such that $G$ is a subgroup of~${\mumu_q \times (\mumu_n \rtimes \mumu_q)}$. Choose the group~${\mumu_q \times (\mumu_n \rtimes \mumu_q)}$ such that its supgroup $\mumu_n$ is also acting on $X$. Suppose $n > 1$. Let $f$ be any lift of the generator of $\mumu_n$. By Lemma \ref{root}, there is an element $g \in \mathbb{K}$ such that $f^n = g^n$ . Then $u = \frac{f}{g}$ is a lift of a generator of~$\mumu_n$ such that~${u^n - 1 = 0}$. Let us make a base change to a separable closure~$\mathbb{K}^{sep}$. Consider fixed points on $X_{\mathbb{K}^{sep}}$. Notice that~$X$ becomes isomorphic to~$\mathbb{P}(V)$, where $V$ is a vector space of dimension $q$. Denote by $\phi$ this isomorphism between $X_{\mathbb{K}^{sep}}$ and $\mathbb{P}(V)$. Then fixed points of $u$ on $X_{\mathbb{K}^{sep}}$ corresponds to eigenvectors of the operator~${u' = \phi u \phi^{-1}}$ on~$V$. Notice that $$(u')^n = (\phi^{-1}u\phi)^n = \phi^{-1}\cdot 1 \cdot \phi = 1.$$Since by Lemma \ref{p=l} one has that $n$ is not divisible by $l$ (if $l > 0$), then polynomial~${u^n - 1}$ is separable and therefore $u'$ is diagonalizable over $\mathbb{K}^{sep}$. Then notice that the fixed by $u$ subvariety of $X_{\mathbb{K}^{sep}}$ of dimension $r$ corresponds to the eigensubspace of the operator $u'$ of dimension $r+1$. Suppose there is a fixed subvariety of dimension at least $1$. Consider one of such a varieties which has the maximal dimension and denote it $Y$. Then let us act on~$Y$ by various~$\sigma \in \Gal(\mathbb{K}^{sep}/\mathbb{K})$. Since action by $u$ was defined over~$\mathbb{K}$ one has that~$u$ and~$\sigma$ commute for all $\sigma \in \Gal(\mathbb{K}^{sep}/\mathbb{K})$, so~$\sigma(Y)$ is again a subvariety fixed by $u$. Note that either $Y = \sigma(Y)$ or $Y$ and $\sigma(Y)$ are disjoint, since $Y$ has a maximal dimension. Consider a twisted-linear subvariety, containing the union of all~$\sigma(Y)$. It is $\Gal(\mathbb{K}^{sep}/\mathbb{K})$-invariant, therefore it is a twisted-linear subvariety defined over~$\mathbb{K}$, but since~$\dim(X) = q - 1$ one has that $X$ is minimal, so the union of all~$\sigma(Y)$ is the whole~$X$. Let us move now by the map~$\phi$ to the vector space $V$, then we have that a union of several disjoint subspaces, i.e., subspaces intersected trivially, of the same dimension $r+1\geqslant 2$ generates a $q$ dimensional vector space, which is impossible since $q$ was a prime number. Therefore, all fixed components are just points. Since points correspond to eigenvectors and there are no eigensubspaces of dimension at least 2, there are exactly $q$ fixed points. 

Denote by $P$ the union of these $q$ points. This point is general since $X$ is a minimal Severi~--~Brauer variety and therefore there is no twisted-linear subvariety containing $P$. Note that~$P$ is defined over $\mathbb{K}$,since $P$ is $\Gal(\mathbb{K}^{sep}/\mathbb{K})$-invariant. Consider a splitting field~$F$ of the polynomial $T^n - 1$ taken as a polynomial over $\mathbb{K}$. Since $T^n - 1$ is a separable polynomial,~$\mathbb{K} \subset F$ is a separable field extension. Note that the point $P$ becomes a union of $q$ different $F$-rational points, therefore $P$ is a separable point. 

Let us prove that $P$ is fixed by the whole group~$G$. Let $\alpha$ be a generator of any $\mumu_q$ which normalize $\mumu_n$. We can consider extension of this action after the base change to $\mathbb{K}^{sep}$. This action will commute with action of the Galois group on $X_{\mathbb{K}^{sep}}$. Notice that~${\alpha^{-1}u\alpha = u^t}$ for some $1 \leqslant t \leqslant n - 1$. Consider $\alpha(x)$ for some $x$ fixed by $u$, compose it with $u$, then $$u\alpha(x) = \alpha(u^t(x)) = \alpha(x).$$ So $\alpha(x)$ is fixed by $u$ and therefore $\alpha$ sends $u$-fixed points to $u$-fixed points.

Now, suppose $n = 1$. Then $G \subset \mumu_q \times \mumu_q$. Let~${\alpha \in G}$ be a generator of $\mumu_q$. Let~$u$ be any lift of~$\alpha$ into $A$. Since $\alpha$ has an order $q$ in $A^*/\mathbb{K}^*$ then there exists $\frak{k} \in \mathbb{K}$ such that~${u^q - \frak{k} = 0}$. Let us prove that $\frak{k} \notin \mathbb{K}^q$. Suppose~${\frak{k} \in \mathbb{K}^q}$, then there is a lift of $\alpha$ into~$A$ of order $q$, i.e. there is~$u \in A$ such that $u^q - 1 = 0$. Consider the following two cases.

Suppose $\Char(\mathbb{K}) = 0$, then by Theorem \ref{NI} one has $$q = \left[\mathbb{K}(u) : \mathbb{K}\right] = \left[\mathbb{Q}(u) : \mathbb{Q}(u) \cap \mathbb{K}\right].$$ It gives a contradiction since $$q-1 = \left[\mathbb{Q}(u) : \mathbb{Q} \right] = \left[\mathbb{Q}(u) : \mathbb{Q}(u) \cap \mathbb{K}\right] \cdot \left[ \mathbb{Q}(u) \cap \mathbb{K} : \mathbb{Q}\right] ,$$ and $q$ cannot divide $q - 1$. 

Suppose $\Char(\mathbb{K}) = l$ a prime number. Then  by Theorem \ref{NI} one has $$q = \left[\mathbb{K}(u) : \mathbb{K}\right] = \left[\mathbb{F}_l(u) : \mathbb{F}_l(u) \cap \mathbb{K}\right].$$ It gives a contradiction since $$ \ord_q(l) = \left[\mathbb{F}_l(u) : \mathbb{F}_l \right] = \left[\mathbb{F}_l(u) : \mathbb{F}_l(u) \cap \mathbb{K}\right] \cdot \left[ \mathbb{F}_l(u) \cap \mathbb{K} : \mathbb{F}_l\right] ,$$ and $q$ cannot divide $q - 1$, but $\ord_q(l)$ is a divisor of $q - 1$.

So $\frak{k} \notin \mathbb{K}^q$. Let us make a base change to separable closure $\mathbb{K}^{sep}$. Then~$X$ becomes isomorphic to $\mathbb{P}(V)$, where $V$ is a vector space of dimension $q$. Denote by $\phi$ this isomorphism between $X_{\mathbb{K}^{sep}}$ and $\mathbb{P}(V)$. Consider an operator $u' = \phi u \phi^{-1}$ on $V$. Since $u^q - \frak{k} = 0$, one has that~${(u')^q - \frak{k} = 0}$ also holds. By Theorem \ref{Leng9}, one has that $(u')^q - \frak{k}$ is an irreducible polynomial, so it is the minimal polynomial of the operator $u'$. Since $q \neq l$, this polynomial is also separable, i.e. it has $q$ different roots. Since $\dim(V) = q$, the degree of the characteristic polynomial of an operator $u'$ is $q$ and it should be divisible by the minimal polynomial, i.e., by $(u')^q - \frak{k}$. So they must be proportional and therefore the characteristic polynomial has $q$ different roots. So there are~$q$ different eigenvectors with different eigenvalues. These eigenvectors corresponds to $q$ different $u$-fixed points on~$X_{\mathbb{K}^{sep}}$ and there are no more $u$-fixed points. Denote the set of these fixed points again by~$P$. By the same argument as in the case $n > 1$, one has that $P$ is a general separable point of degree~$q$ fixed by $G$. And again since $u$ and~$\Gal(\mathbb{K}^{sep}/\mathbb{K})$ commute~$P$ is defined over $\mathbb{K}$.
\end{proof}

Now we are ready to prove the main result of the appendix. 

\begin{thm}
Let $X$ be a non-trivial Severi--Brauer variety over a field $\mathbb{K}$ of dimension~$q - 1$, where $q \geqslant 3$ is a prime number and $\Char(\mathbb{K}) \neq q$. Let $G$ be a finite group acting on $X$. Then $X$ is not $G$-birationally rigid.
\end{thm}

\begin{proof}
Let $A$ be a central simple algebra corresponding to $X$. Denote by $[A]$ the class of $A$ in Brauer group $\Br(\mathbb{K})$. By Proposition \ref{lembir}, there is a~$G$-fixed general separable point of degree $q$. Then one can make a Cremona transformation in this point from $X$ to another Severi--Brauer variety $Y$. Let $B$ be a central simple algebra corresponding to $Y$. Since Cremona transformation becomes over $\mathbb{K}^{sep}$ a rational map defined by polynomials of degree $q - 1$, then by \cite[Theorem 3.3.7(iii)]{GSh} one has that~${[B] = (q-1)[A]}$. By \cite[~Theorem~4.5.13]{G-S}) one has that $q[A] = 0$, therefore $Y$ corresponds to an opposite algebra~$A^{op}$ of $A$. Since $X$ is non-trivial, then $[A]$ is a non-trivial class in Brauer group~$\Br(\mathbb{K})$. Since $$[A] + [A^{op}] = 0$$ and $$q[A] = 0$$ for the prime number $q \geqslant 3$, then $[A] \neq [A^{op}]$. Therefore, $X$ is not isomorphic to $Y$ and so $X$ is not $G$-birationally rigid. 
\end{proof}

\end{document}